\documentclass[letterpaper, 10 pt, conference]{ieeeconf}
\IEEEoverridecommandlockouts
\usepackage{amssymb} 
\usepackage{amsmath}
\usepackage{tikz} 
	\usetikzlibrary{calc}
	\usetikzlibrary{shapes}
\usepackage{mathtools} 
\usepackage{caption} 
\usepackage{subcaption}

\newtheorem{thm}{Theorem}
\newtheorem{lemma}{Lemma}

\newtheorem{defn}{Definition}

\newtheorem{asmp}{\textbf{Assumption}}
\newtheorem{remark}{Remark}

\newcommand{\reals}{\mathbb{R}}

\newcommand{\defeq}{\stackrel{\Delta}{=}}
\newcommand{\norm}[1]{\left \lVert#1\right \rVert}
\newcommand{\snorm}[1]{\left \lVert#1\right \rVert^{2}}
\newcommand{\set}[1]{\left\{#1\right\}}
\newcommand{\setc}[2]{\{#1\ |\ #2\}}

\newcommand{\bracket}[1]{\left( #1 \right)}

\DeclareMathOperator{\col}{col}
\DeclareMathOperator{\row}{row}

\DeclarePairedDelimiterX{\inp}[2]{\langle}{\rangle}{#1\ \middle| \ #2}
\DeclarePairedDelimiterX{\inpl}[2]{\langle}{.}{#1\ \middle| \ #2}
\DeclarePairedDelimiterX{\inpr}[2]{.}{\rangle}{#1\ \ #2}

\tikzstyle{block} = [draw, fill=white, rectangle, 
    minimum height=2em, minimum width=2em]
\tikzstyle{sum} = [draw, fill=white, circle, node distance=0.5cm, inner sep=0pt, minimum size=0.25cm]
\tikzstyle{input} = [coordinate]
\tikzstyle{output} = [coordinate]
\tikzstyle{pinstyle} = [pin edge={to-,thin,black}]
\tikzstyle{branch}=[fill,shape=circle,minimum size=3pt,inner sep=0pt]

\usepackage[ruled,vlined]{algorithm2e}
\SetKwProg{init}{Init}{}{}

\usepackage{multirow}
\usetikzlibrary{cd}

\newcounter{num}

\def\num{\par\medskip\refstepcounter{num}\hangindent2em{\bfseries\arabic{num}}.\hspace{1em}\MakeUppercase}
\newenvironment{Numera}
{\parindent0pt\par\medskip}
{\setcounter{num}{0}\par\bigskip}

\title{Resilient Nash Equilibrium Seeking in the Partial Information Setting}

\author{Dian Gadjov and Lacra Pavel
\thanks{This work was supported by an NSERC Discovery Grant.}%
\thanks{\footnotesize D. Gadjov and L. Pavel are with Dept. of Electrical and Computer Engineering, University of Toronto, 
        {\tt\footnotesize dian.gadjov@mail.utoronto.ca},
        {\tt\footnotesize pavel@control.utoronto.ca}}%
}

\begin{document}

\maketitle
\thispagestyle{empty}
\pagestyle{empty}

\begin{abstract}
  Current research in distributed Nash equilibrium (NE) seeking in the partial information setting assumes that information is exchanged between agents that are ``truthful''. However, in general noncooperative games agents may consider sending misinformation to neighboring agents with the goal of further reducing their cost. Additionally, communication networks are vulnerable to attacks from agents outside the game as well as communication failures. In this paper, we propose a distributed NE seeking algorithm that is robust against adversarial agents that transmit noise, random signals, constant singles, deceitful messages, as well as being resilient to external factors such as dropped communication, jammed signals, and man in the middle attacks. The core issue that makes the problem challenging is that agents have no means of verifying if the information they receive is correct, i.e. there is no ``ground truth''. To address this problem, we use an observation graph, that gives truthful action information, in conjunction with a communication graph, that gives (potentially incorrect) information. By filtering information obtained from these two graphs, we show that our algorithm is resilient against adversarial agents and converges to the Nash equilibrium. 
\end{abstract}

\section{Introduction}
  Designing distributed NE seeking algorithms is currently an active research area. This is due the wide range of problems that can be formulated as a network scenario between self-interested agents. A small sample of problems that have been formulated as a distributed NE problem are: demand-side management for smart grids, \cite{Basar2012}, electric vehicles, \cite{PEVParise}, competitive markets, \cite{LiDahleh}, network congestion control, \cite{Garcia}, power control and resource sharing in wireless/wired peer-to-peer networks, cognitive radio systems, \cite{Scutari_2014}, etc.

  Classically, Nash equilibrium problems were solved for the \emph{full-decision information} setting, i.e., whereby all players have knowledge of every agent's decision/action \cite{FP07}. This requires a centralized system to disseminate the action information \cite{gramDR} or agents must transmit their action information to all other agents. However, there is a large class of problems where this assumption is unreasonable. For example, either when a central coordinator does not exist, or it is infeasible to have a centralized system, e.g., due to cost, computation burden, agent limitations, etc. In recent years, research has focused on the  \textit{partial-decision information} setting, where agents do not know the actions of the other agents, but may communicate with neighbouring agents \cite{dianCT}, \cite{nedich}, \cite{gram17}. Agents communicate to learn the true value of the actions of all agents in the network.

  All such existing algorithms make the implicit assumption that agents share information truthfully. For some problems, this assumption is reasonable because the agents are working for a common goal, the agents are guided by a central authority, or there are systems in place so that sending misinformation is impossible. Even in these settings, the agents are susceptible to external factors, such as communication failures, and communication attacks. Moreover, in general non-cooperative games, agents might have an incentive to not be truthful if deceiving others can minimize their cost further. Motivated by this problem, we propose a NE seeking algorithm that is robust against adversarial agents and communication failures/attacks.

In the consensus problem literature, there are algorithms designed to reach a consensus in settings where there are malicious agents under a wide variety of communication assumptions \cite{robust_consensus_0}, \cite{robust_consensus_1}, \cite{robust_consensus_2}, \cite{robust_consensus_3}, \cite{robust_consensus_4}, \cite{robust_consensus_5}, \cite{resilient_sensor_attack}. However, these algorithms are solely focused on just solving the consensus problem, which is just a component of the NE seeking process. These algorithms do not account for the optimization problem that each agent is trying to solve.

In the distributed optimization literature, ideas from the consensus problem have been applied to design resilient optimization algorithms. However, agents only converge to a convex combination of the local function minimizers and are unable to guarantee convergence to the optimal solution \cite{robust_opt}, \cite{robust_opt_ghar}, \cite{res_opt_2}, \cite{res_opt_3}. The reason for this shortcoming is that adversarial agents are indistinguishable from a normal agent with a modified (and valid) cost function \cite[Theorem 4.4]{robust_opt_ghar}.

In the game theory literature, malicious agents have been studied in various settings \cite{malice}, \cite{deception}. However, the focus of the research is on understanding the behavior of agents or designing mechanisms to remove the incentive to act maliciously. In the distributed NE seeking literature there are very few results that attempt to tackle this problem. The authors in \cite{resilient_NE}, propose a NE seeking algorithm against denial of service (DOS) attacks. However, the paper makes assumptions on the frequency and duration of attacks, as well as assuming that agents can detect when a communication link is attacked. While \cite{resilient_NE} does handle a specific external attack on the network, it does not deal with malicious agents in the game.

\textit{Contributions. } In this paper, we propose a novel distributed NE seeking algorithm that is robust to adversarial agents. The algorithm is resilient to both external factors, such as attacks from agents from outside the game and communication faults, and internal adversaries that transmit messages to neighboring agents with the intention of deceiving them to reduce their own cost. To achieve this we utilize a so called observation graph through which actions can be directly observed and is not susceptible to being tampered with. To the best of our knowledge this is the first such result in the literature.
As compared with \cite{cdc_version} this paper relaxes some of the assumptions, provides proofs for all results that were absent in \cite{cdc_version} due to space limitations, and a larger scale example.

The paper is organized as follows. Section \ref{sec:background} gives preliminary background. Section \ref{sec:formulation} formulates the problem. In Section \ref{sec:obs_graph} we introduce the observation graph and discuss properties of graphs. The proposed algorithm is presented in Section \ref{sec:algo} and the convergence analysis is presented in Section \ref{sec:conv}, numerical results are provided in Section \ref{sec:sim}, and concluding remarks are given in Section \ref{sec:conclusion}.

\subsection{Notation and Terminology} \label{sec:background}

Let $(\mathbf{1}_{n})$ $\mathbf{0}_{n}$ denote the vector of all (ones) zeros of dimension $n$, and $I_n$ the $n \times n$ identity matrix. To ease notation, we will drop the subscript when the dimension can be inferred. Given a set $C\subset S$ let $\bar{C} = S \setminus C$ denote the complement of set $C$. Given an ordered index set $\mathcal{I} = \set{1,2,\dots, n}$ let $\col(x_{i})_{i\in\mathcal{I}} \defeq [x_{1}^{T},x_{2}^{T},\dots,x_{n}^{T}]^{T}$ and $\row(x_{i})_{i\in\mathcal{I}} \defeq [x_{1}^{T},x_{2}^{T},\dots,x_{n}^{T}]$. Given matrices $A$ and $B$, we use $A \succ B$ ($A \succeq B$) to denote $A-B$ is positive definite (positive semidefinite). We denote the Euclidean norm of $x$ as $\norm{x}^2 \defeq \inp*{x}{x}$. For a given symmetric matrix $M\succeq 0$, let  $\inp*{x}{y}_{M} \defeq \inp*{x}{My}$ and  $\snorm{x}_{M} \defeq \inp*{x}{x}_{M}$, which is a norm when $M  \succ 0$. 
 
A directed graph is denoted by $\mathcal{G} = (\mathcal{N},\mathcal{E})$ where $\mathcal{N}$ is the set of nodes and $\mathcal{E} \subset \mathcal{N}\times \mathcal{N}$ is the set of edges. Let $e_{ij}\in \mathcal{E}$ denote an edge from node $j$ to node $i$. For graph $\mathcal{G}$, the set of in-neighbours to node $i$ is denoted $\mathcal{N}_{i}^{in}(\mathcal{G}) \defeq \setc{j}{e_{ij}\in\mathcal{E}}$ and the set of out-neighbours is $\mathcal{N}_{i}^{out}(\mathcal{G}) \defeq \setc{j}{e_{ji}\in\mathcal{E}}$. Additionally, the set of in-edges to node $i$ is denoted $\mathcal{E}_{i}^{in}(\mathcal{G}) \defeq \setc{e_{ij}}{e_{ij}\in\mathcal{E}}$ and the set of out-edges is $\mathcal{E}_{i}^{out}(\mathcal{G}) \defeq \setc{e_{ji}}{e_{ji}\in\mathcal{E}}$. A path from node $i$ to node $j$ is a sequence of nodes $v_{1},v_{2},\dots,v_{n}$ such that $v_{1}=i$, $v_{n}=j$ and $e_{v_{k+1},v_{k}}\in\mathcal{E}$ $\forall k\in \set{1,2,\dots,n-1}$. A graph is rooted at $i$ if there is a path from $i$ to all $j\in\mathcal{N}\setminus \set{i}$. A graph is strongly connected if there is a path from every node to every other node.

\section{Problem Formulation} \label{sec:formulation}

Consider a set of $N$ players/agents denoted $\mathcal{N}=\set{1,2,\dots,N}$. Each player $i\in\mathcal{N}$ selects an action $x_{i} = \col(x_{iq})_{q\in\set{1,2,\dots,n_{i}}}\in\Omega_{i} \subset \reals^{n_{i}}$, where $\Omega_{i}$ is the constraint set and $x_{iq}\in\reals$ is the $q^{th}$ component of agent $i$'s action. Let $x \defeq \col(x_{i})_{i\in\mathcal{N}}\in\Omega$ denote the action profile of all the players actions, where $\Omega \defeq \prod_{i\in\mathcal{N}}\Omega_{i} \subset \reals^{n}$ and $n = \sum_{i\in\mathcal{N}} n_{i}$. Additionally, we will denote $x = (x_{i},x_{-i})$ where $x_{-i}\in\Omega_{-i}=\prod_{j\in\mathcal{N}\setminus \set{i}}\Omega_{j}$ is the action of all players except for $i$. Each agent $i$ has a cost function $J_{i}: \Omega \to \reals$.  which we denote as $J_{i}(x_{i},x_{-i})$ to emphasize that the cost is dependent on $i$'s own and the other agents actions.

In a game $G$, agents have full information (FI) if for all agent $i$ knows $x_{-i}$ and agents have partial information (PI) otherwise. When agents have full information the problem that each agent tries to solve is,
\begin{align} \label{ne-fi}
\begin{split}
	\min_{x_{i}}&\quad J_{i}(x_{i},x_{-i})\\
	s.t. &\quad x_{i}\in\Omega_{i}
\end{split} \tag{NE-FI}
\end{align}

\begin{defn}
	Given a game $G$, an action profile $x^{*}=(x_{i}^{*}, x_{-i}^{*})$ is a Nash Equilibrium (NE) if $x^{*}$ solves \eqref{ne-fi} for all agents $i\in\mathcal{N}$. \hfill \QEDopen
\end{defn}

We make the following standard assumption used in distributed NE seeking problems, \cite{dianCT}, \cite{gram17}, that ensures that the NE of the game exists and is unique.
\begin{asmp} \label{asmp:f_and_F_conditions} The set $\Omega = \reals^{n}$. For each player $i\in\mathcal{N}$, given any $x_{-i}$, $J_{i}(x_{i}, x_{-i})$ is continuously differentiable, convex in $x_{i}$ and radially unbounded in $x_{i}$. The pseudo-gradient $F(x)$ is $\mu$-strongly monotone and $L$-Lipschitz, where $F(x) \defeq \col\bracket{\frac{\partial J_{i}(x)}{\partial x_{i}}}_{i\in\mathcal{N}}$. \hfill \QEDopen
\end{asmp}

\begin{remark} \label{remark:ne_char}
	Under Assumption \ref{asmp:f_and_F_conditions} the NE can be characterized as the action profile $x^{*}$ where the pseudo-gradient is $\mathbf{0}$, i.e., $F(x^{*}) = \mathbf{0}$.
\end{remark}

For Nash equilibrium problems, it is classically assumed that agents know the other agents' actions, i.e., agent $i$ knows $x_{-i}$ \cite{FP07}. However, in many scenarios this assumption is impractical. For example, consider a scenario where drones are ordered to arrange themselves into a certain configuration. It may be unreasonable to assume that each drone can know where every other drone is, since the distances between agents could be quite large.

If agents only have partial information then they cannot solve \eqref{ne-fi} because they do not know $x_{-i}$ fully. One way to try and resolve this issue is to have agents share information to neighboring agents over a communication graph $\mathcal{G}_{c}$. In the drone example, this could mean that drones have a transmitter that can only broadcast information over a small radius such that only nearby agents can receive this message. 

\begin{defn}
	A graph $\mathcal{G}_{c} = (\mathcal{N}, \mathcal{E}_{c})$ is a communication graph where the node set $\mathcal{N}$ is the set of agents and the edge set $\mathcal{E}_{c}$ describes the flow of information between agents. An edge $e_{ij} \in \mathcal{E}_{c}$ if agent $j$ gives information to agent $i$. \hfill \QEDopen
\end{defn}


In the partial information case, agent $i$ does not know $x_{-i}$ and uses $\mathcal{G}_{c}$ to learn $x_{-i}$. To do so, agent $i$ maintains a vector $\mathbf{x}^{i}\defeq\col(\mathbf{x}^{i}_{j})_{j\in\mathcal{N}}\in\reals^{n}$ where $\mathbf{x}^{i}_{j}\in\reals^{n_{j}}$ is agent $i$'s \emph{genuine} estimate/belief of agent $j$'s action. Furthermore, $\mathbf{x}^{i}_{j} = \col(\mathbf{x}^{i}_{jq})_{q\in \set{1,2,\dots,n_{j}}}$ where $\mathbf{x}^{i}_{jq}\in\reals$ is agent $i$'s genuine belief of the $q^{th}$ component of agent $j$'s action. Note that $\mathbf{x}^{i}_{i}$ is agent $i$'s estimate of their own action, since they know their own action $\mathbf{x}^{i}_{i} = x_{i}$. Ideally, all agents $i$ would like to use the communication graph to update their beliefs such that $\mathbf{x}^{i} = x$, i.e., their belief matches reality. We denote all the agents' estimates/beliefs stacked into a single vector as $\mathbf{x} \defeq \col(\mathbf{x}^{i})_{i\in\mathcal{N}}$. We denote the \emph{communicated message} that agent $j$ sends to agent $i$, via $\mathcal{G}_{c}$, about the action of agent $m$ as $\mathbf{y}^{ij}_{m} \in \reals^{n_{m}}$. Furthermore, $\mathbf{y}^{ij}_{m} = \col(\mathbf{y}^{ij}_{mq})_{q\in \set{1,2,\dots,n_{m}}}$ where $\mathbf{y}^{ij}_{mq}\in\reals$ is agent $j$'s communicated message to agent $i$ about the $q^{th}$ component of agent $m$'s action. Additionally, let $\mathbf{y}^{ij} = \col(\mathbf{y}^{ij}_{m})_{m\in\mathcal{N}} \in \reals^{n}$, $\mathbf{y}^{i} = \col(\mathbf{y}^{ij})_{j\in\mathcal{N}} \in \reals^{Nn}$, and without loss of generality $\mathbf{y}^{ii} = \mathbf{x}^{i}$, i.e., the message agent $i$ ``sends to himself'' is equal to agent $i$'s genuine belief. In general, agent $i$ will update its estimate/actions and communicate messages via,
\begin{align}\label{eqn:dyn_general}
\begin{split}
	\mathbf{x}^{i}[k+1] &= f_{i}(\mathbf{x}^{i}[k],\mathbf{y}^{i}[k])\\
	\mathbf{y}^{ji}[k+1] &= h_{ji}(\mathbf{x}^{i}[k],\mathbf{y}^{i}[k]) \quad \forall j\in \mathcal{N}_{i}^{out}(\mathcal{G}_{c})
\end{split}
\end{align}
where $[k]$ denotes the $k^{\text{th}}$ iteration of the algorithm. The function $f_{i}$ describes how agent $i$ updates its estimate, while $h_{ji}$ describes what information agent $i$ communicates to agent $j$. 

In the partial information setting agent $i$ can not solve \eqref{ne-fi} because they do not know $x_{-i}$. Therefore, consider the following problem for the partial information setting,
\begin{align} \label{ne-pi-t}
\begin{split}
	\min_{\mathbf{x}^{i}_{i}}&\quad J_{i}(\mathbf{x}^{i}_{i}, \mathbf{x}^{i}_{-i})\\
	s.t. &\quad \mathbf{x}^{i}_{i} \in\Omega_{i} \\
	&\quad \mathbf{x}^{i} = \mathbf{y}^{ij}\quad \forall j\in\mathcal{N}_{i}^{in}(\mathcal{G}_{c}) \\
	&\quad \mathbf{y}^{ji} = \mathbf{x}^{i}\quad \forall j\in\mathcal{N}_{i}^{out}(\mathcal{G}_{c})
\end{split} \tag{NE-PI-T}
\end{align}

\begin{remark} \label{remark:ne_partial_char}
	Similar to the full-decision information setting case (Remark \ref{remark:ne_char}), the NE can be characterizes as the action profile $\mathbf{x}^{*}$ where the \emph{extended pseudo-gradient} $\mathbf{F}(\mathbf{x}) \defeq \col\bracket{\frac{\partial J_{i}(\mathbf{x}^{i})}{\partial x_{i}}}_{i\in\mathcal{N}}$ is $\mathbf{0}$ and $\mathbf{x}$ is at consensus, e.g., $\mathbf{F}(\mathbf{x}^{*}) = \mathbf{0}$ and $\mathbf{x}^{*} = \mathbf{1}_{N}\otimes x^{*}$ where $x^{*}$ is the NE.
\end{remark}

\begin{lemma} \label{lemma:partial_opt}
	Given a game $G$, let $x^{*}$ be a solution to \eqref{ne-fi} for all agents $i\in\mathcal{N}$, i.e., $x^{*}$ is the NE. Assume that $\mathcal{G}_{c}$ is a strongly connected directed graph. If $\bar{\mathbf{x}} = (\bar{\mathbf{x}}^{1}, \cdots, \bar{\mathbf{x}}^{N})$ solves \eqref{ne-pi-t} for all agents $i\in\mathcal{N}$ then $\bar{\mathbf{x}} = \mathbf{1}_{N}\otimes x^{*}$ where $x^{*}$ and $\mathbf{x}^{i}_{i} = x^{*}_{i}$.
\end{lemma}
\begin{proof}
	Since $\mathcal{G}_{c}$ is a strongly connected directed graph, the last two constraints in \eqref{ne-pi-t} are equivalent to the condition that $\mathbf{x}^{i}=\mathbf{x}^{j}$ for all $i,j \in \mathcal{N}$. Using the fact that $\mathbf{x}^{i}_{i} = x_{i}$, this implies that $\mathbf{x}^{i}=\mathbf{x}^{j} = x$ for all $i,j \in \mathcal{N}$. Therefore, we can replace $\mathbf{x}^{i}_{i}$ with $x_{i}$ and $\mathbf{x}^{i}_{i}$ with $x_{-i}$ in \eqref{ne-pi-t} which gives \eqref{ne-fi}.
\end{proof}

\begin{remark}
Lemma \ref{lemma:partial_opt} shows that if each agent solves \eqref{ne-pi-t} then they will reach the NE. Notice that \eqref{ne-pi-t} introduces two additional constraints as compared with \eqref{ne-fi} which correspond to constraints on the communication between agents. The constraint $\mathbf{x}^{i} = \mathbf{y}^{ij}$ means that $i$'s estimate is in agreement with the communicated message of all its in-neighbours. This can be interpreted as a local consensus constraint. The constraint $\mathbf{y}^{ji} = \mathbf{x}^{i}$ means agent $i$'s communicated message to his out-neighbours is the same as its estimate, that is, agent $i$ is communicating honestly with its neighbours.
\end{remark}

While ideally we would like all agents to solve \eqref{ne-pi-t}, a self-interested agent has an incentive to ignore the second constraint (communicating honestly with its neighbours) if it would result in lowering their cost further. Thus, a self-interested agent would instead solve the following modified optimization problem.
\begin{align} \label{ne-pi-a}
\begin{split}
	\min_{\mathbf{x}^{i}_{i}}&\quad J_{i}(\mathbf{x}^{i}_{i},\mathbf{x}^{i}_{-i})\\
	s.t. &\quad x_{i}\in\Omega_{i} \\
	&\quad \mathbf{x}^{i}_{-i} = \mathbf{y}^{ij}_{-i}\quad \forall j\in \mathcal{N}_{i}^{in}(\mathcal{G}_{c}) 
\end{split} \tag{NE-PI-A}
\end{align}
Note, that \eqref{ne-pi-a} is obtained from \eqref{ne-pi-t} by removing the second constraint and modifying the first constraint. The constraint in \eqref{ne-pi-a} ensures that their belief about other agent's actions matches what their neighbors communicate. If the agent is lying about their own action then they do not want their action to match the lie, hence why $\mathbf{x}^{i}_{-i} = \mathbf{y}^{ij}_{-i}$ and not $\mathbf{x}^{i} = \mathbf{y}^{ij}$.

\begin{defn}
	Agent $i$ is truthful if it is solving \eqref{ne-pi-t} at each iteration $k$, and the data agent $i$ is communicating is
\begin{align} \label{eqn:truth_comm}
	\mathbf{y}^{ji}[k] &= \mathbf{x}^{i}[k]\quad \forall j \in \mathcal{N}_{i}^{out}(\mathcal{G}_{c})
\end{align}
	The agent is adversarial if it is solving \eqref{ne-pi-a} and/or  if the data it communicates does not satisfy  \eqref{eqn:truth_comm}. \hfill \QEDopen
\end{defn}
Assuming that truthful agents solve \eqref{ne-pi-t} and adversarial agents solve \eqref{ne-pi-a} only models agents in the game. It does not model external factors such as man in the middle attacks, broken communication devices, signal jammers, etc. While an agent might intend to be truthful, these external factors give the appearance that the agent is adversarial. The addition of the communication dynamics \eqref{eqn:truth_comm} is used to classify these external factors as adversarial. For example, consider that agent $i$ is solving \eqref{ne-pi-t} but $\mathbf{y}^{ji}[k] = c\mathbf{1}$, $\forall k=\set{0,1,\dots}$. This example could represent that agent $i$'s communication device is broken. 

\noindent\rule[0.5ex]{\linewidth}{1pt}
\textbf{Objective:} Given a game $G$, let $x^{*}$ denote the NE, and let $\mathcal{N} = \mathcal{T} \cup \mathcal{A}$ where $\mathcal{T}$ is the set of truthful agents and $\mathcal{A}$ is the set of adversarial agents. Design an algorithm such that $\mathbf{x}^{i}_{i} \to x^{*}_{i}$ for all $i\in\mathcal{N}$, when the set $\mathcal{A} \neq \emptyset$. 

\noindent\rule[0.5ex]{\linewidth}{1pt}
\vspace{-5mm}

In the literature, typical distributed NE seeking algorithms solve the problem where $\mathcal{A} = \emptyset$ and all agents satisfy \eqref{eqn:truth_comm}. This is the standard formulation for distributed NE seeking problems and by Lemma \ref{lemma:partial_opt} the solution is the NE.

\section{Observation, Robust, and Information Robust Graphs} \label{sec:obs_graph}

In this section, we first discuss the issues of the current problem formulation and then introduce the observation graph to deal with this problem. The remainder of the section will present various graph notions that are critical in the analysis of our proposed algorithm in the following section.

In the robust distributed optimization literature, it has been shown that in general it is impossible to detect an adversarial agent \cite[Theorem 4.4]{robust_opt_ghar}. This is because an adversarial agent's behavior is indistinguishable from a truthful agent with a modified (and valid) cost function. 

For the NE seeking problem in the partial information setting, there is additional information that could be used to detect adversarial agents. For example, agent $i$ receives a cost $J_{i}(x_{i},x_{-i})$ which depends on the actual actions played. If the cost received does not match the cost using the estimated actions, $J_{i}(x_{i},x_{-i}) \neq J_{i}(\mathbf{x}^{i}_{i},\mathbf{x}^{i}_{-i})$, then agent $i$ knows that the information they received is wrong. Therefore, agents would be able to detect that there is an adversarial agent, but they would not be able to identify which one of their opponents is adversarial\footnote{Unless there is only $1$ other agent or the game has strong assumptions on the cost functions and the communication graph}. 

The core issue is that agents have no way of verifying if the information they receive is correct. Additionally, if the information is determined to be incorrect there is no mechanism to get the correct information. Unfortunately, \emph{without additional assumptions it is impossible to solve this problem}. In order to deal with these issues, we assume that each agent $i$ can directly observe/measure some of actions $x_{j}$, without communication over $\mathcal{G}_{c}$.

\begin{defn}
	An \emph{observation graph} $\mathcal{G}_{o} = (\mathcal{N},\mathcal{E}_{o})$ describes what actions can be directly observed\footnote{In our setting, agent $i$ knows what action they used and therefore $e_{ii} \in \mathcal{E}_{o}$ for all $i\in\mathcal{N}$.}. If edge $e_{ij}\in \mathcal{E}_{o}$ then agent $i$ can directly observe/measure $x_{j}$.
	 \hfill \QEDopen
\end{defn}

Note that agents now get information from two different sources. Agents get information via the communication graph $\mathcal{G}_{c}$, which contains potentially incorrect information about \emph{all} agents actions, and the observation graph $\mathcal{G}_{o}$, which contains always correct \emph{local} information about neighbouring agents' actions.

  Continuing with the drone example, assume that the position of nearby drones can be measured (observation graph). The position of distant drones can not be measured directly, but the position information for those drones is obtained through a communication graph, which can be susceptible to  communication failures and signal interference. 

Unless $\mathcal{G}_{o}$ is a complete graph, agents will not fully know $x_{-i}$, i.e., agent $i$ does not know the actions of agents $\mathcal{N}\setminus \mathcal{N}_{i}^{in}(\mathcal{G}_{o})$. For agent $i$ to learn the actions of agents who are not observed, they will have to rely on the communication graph. Even if agents can directly observe a subset of the agents' actions the problem is not trivial. For example, consider that both $\mathcal{G}_{c}$ and $\mathcal{G}_{o}$ are given by the graph in Figure \ref{fig:Graph_O_asmp_not_trival}.

\vspace{2mm}
{\centering
\begin{minipage}[ht]{1\columnwidth}
\center
\begin{tikzpicture}
\def\Nnodes{3}
\def\radius{15mm}

\node[draw, circle, green!65!black] (n1) at (0,0) {1};
\node[draw, circle, red!90!black] (n2) at (1,1) {2};
\node[draw, circle, green!65!black] (n3) at (1,-1) {3};
\node[draw, circle, green!65!black] (n4) at (2,0) {4};
	
\path[<-,line width=0.20mm]
	(n1) edge (n2)
	(n1) edge (n3)
	(n2) edge (n4)
	(n3) edge (n4)
;

\end{tikzpicture}

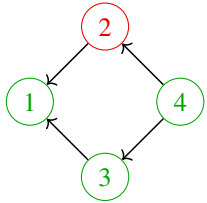
\captionof{figure}{$\mathcal{G}_{c}=\mathcal{G}_{o}$, adversary can not be identified}
\label{fig:Graph_O_asmp_not_trival}
\end{minipage}}
\vspace{0mm}

The red (green) node in the graph represents the adversarial (truthful) agent. For this observation graph agent $1$ can measure $x_{2}$ and $x_{3}$, agent $2$ can measure $x_{4}$, and agent $3$ can measure $x_{4}$. Therefore, agent $1$ does not know the action of agent $4$ but agent $1$'s neighbours know the action of agent $4$. Using the communication graph $\mathcal{G}_{c}$, agent $1$ will receive two messages $y^{1,2}_{4}$ and $y^{1,3}_{4}$ about agent $4$'s action, one truthful message and the other one an adversarial message. It is not obvious how agent $1$ can discern which of the two messages is truthful.

In the following, we analyze the properties of both the communication and observation graph. Under appropriate conditions on these two graphs we will later prove convergence of the proposed algorithm. The following definitions are used to describe how connected a subset of nodes to the rest of the graph is, and how connected a graph is.

\begin{defn}[ Def 4, \cite{robust_consensus_0}]
Given a graph $\mathcal{G}=(\mathcal{N}, \mathcal{E})$, a subset $S\subset \mathcal{N}$ is \emph{$r$-local} if $(\forall i\in \bar{S})$, $|\mathcal{N}^{in}_{i}(\mathcal{G})\cap S| \leq r$. \hfill \QEDopen
\end{defn}
	If the adversarial agent set $\mathcal{A}$ is $r$-local, then this means that the number of adversarial neighbours to a truthful agent $i\in\mathcal{T}$ is less than or equal to $r$. For example, in Figure \ref{fig:r_local_example} the $a_{i}$ node are elements in $\mathcal{A}$ and the numbered nodes are elements in $\mathcal{T}$. The number represents how many in-neighbours are in $\mathcal{A}$. In this example we can see that the most surrounded node has $3$ neighbours in $\mathcal{A}$, therefore the set $\mathcal{A}$ is $3$-local.

\vspace{2mm}
\noindent
\begin{minipage}[ht]{1\columnwidth}
\centering
\begin{tikzpicture}
\def\Nnodes{3}
\def\radius{15mm}

\node[draw, circle] (n1) at (0,0) {0};
\node[draw, circle] (n2) at (1,0) {0};
\node[draw, circle] (n3) at (2,0) {1};
\node[draw, circle, red!90!black] (n4) at (3,0) {$a_{1}$};
\node[draw, circle] (n5) at (4,0) {2};
\node[draw, circle] (n6) at (5,0) {1};
\node[draw, circle, red!90!black] (n7) at (6,0) {$a_{2}$};
\node[draw, circle] (n8) at (7,0) {1};
\node[draw, circle] (n9) at (0,1) {0};
\node[draw, circle] (n10) at (1,1) {1};
\node[draw, circle] (n11) at (2,1) {0};
\node[draw, circle, blue] (n12) at (3,1) {3};
\node[draw, circle, red!90!black] (n13) at (4,1) {$a_{3}$};
\node[draw, circle] (n14) at (5,1) {1};
\node[draw, circle] (n15) at (6,1) {2};
\node[draw, circle] (n16) at (7,1) {0};
\node[draw, circle] (n17) at (0,2) {1};
\node[draw, circle, red!90!black] (n18) at (1,2) {$a_{4}$};
\node[draw, circle] (n19) at (2,2) {2};
\node[draw, circle, red!90!black] (n20) at (3,2) {$a_{5}$};
\node[draw, circle] (n21) at (4,2) {2};
\node[draw, circle] (n22) at (5,2) {1};
\node[draw, circle, red!90!black] (n23) at (6,2) {$a_{6}$};
\node[draw, circle] (n24) at (7,2) {1};

\path[-,line width=0.20mm]
	(n1) edge (n2)
	(n2) edge (n3)
	(n3) edge (n4)
	(n4) edge (n5)
	(n5) edge (n6)
	(n6) edge (n7)
	(n7) edge (n8)
	(n9) edge (n10)
	(n10) edge (n11)
	(n11) edge (n12)
	(n13) edge (n14)
	(n14) edge (n15)
	(n15) edge (n16)
	(n17) edge (n18)
	(n18) edge (n19)
	(n19) edge (n20)
	(n20) edge (n21)
	(n21) edge (n22)
	(n22) edge (n23)
	(n23) edge (n24)
	(n1) edge (n9)
	(n2) edge (n10)
	(n3) edge (n11)	
	(n5) edge (n13)
	(n6) edge (n14)
	(n7) edge (n15)
	(n8) edge (n16)
	(n9) edge (n17)
	(n10) edge (n18)
	(n11) edge (n19)	
	(n13) edge (n21)
	(n14) edge (n22)
	(n15) edge (n23)
	(n16) edge (n24)
;

\path[-, line width=0.4mm]
	(n12) edge[blue] (n13)
	(n4) edge[blue] (n12)
	(n12) edge[blue] (n20)
;

\end{tikzpicture}
\captionof{figure}{Example of $\mathcal{A}$ being $3$-local}
\label{fig:r_local_example}
\end{minipage}
\vspace{0mm}

\begin{defn}
	Given a communication graph $\mathcal{G}_{c} = (\mathcal{N},\mathcal{E}_{c})$ and an observation graph $\mathcal{G}_{o} = (\mathcal{N},\mathcal{E}_{o})$, we say that node $i\in\mathcal{N}$ is \emph{$r$-information robust}\footnote{Node $i$ being $r$-information robust is related to the notion of an $r$-robust graph in \cite{robust_consensus_0}} if $(\forall S \supset \mathcal{N}^{out}_{i}(\mathcal{G}_{o}))$ $(\exists j\in \bar{S})$ such that $|\mathcal{N}^{in}_{j}(\mathcal{G}_{c})\cap S| \geq r$. \hfill \QEDopen
\end{defn}
A node $i$ being $r$-information robust describes how many paths are from node $i$ to every other node. For example, in Figure \ref{fig:i_is_r_info_robust} we construct any set $S$ (nodes in the rectangle) that contains node $i$ and agents that can directly observe $i$, i.e., nodes $o_{j}$. You can interpret the set $S$ as all nodes $s\in S$ that have a path from $i$ to $s$. Node $i$ is $r$-information robust if there is a node $j$ outside of $S$ that has at least $r$ edges connecting into $S$. In Figure \ref{fig:i_is_r_info_robust} node $i$ is $2$-information robust and since there is a node $j$ that has at least $2$ edges connecting into $S$.

\vspace{2mm}
\noindent
\begin{minipage}[ht]{1\columnwidth}
\centering
\begin{tikzpicture}
\def\Nnodes{3}
\def\radius{8mm}

\node[draw, circle, minimum size=\radius] (n1) at (0,0) {};
\node[draw, circle, minimum size=\radius] (n2) at (1,0) {};
\node[draw, circle, minimum size=\radius] (n3) at (2,0) {};
\node[draw, circle, minimum size=\radius] (n4) at (3,0) {$\color{magenta} j$};
\node[draw, circle, minimum size=\radius] (n5) at (4,0) {};
\node[draw, circle, minimum size=\radius] (n6) at (5,0) {};
\node[draw, circle, minimum size=\radius] (n7) at (6,0) {};
\node[draw, circle, minimum size=\radius] (n8) at (7,0) {};
\node[draw, circle, minimum size=\radius] (n9) at (0,1) {$\color{blue!90!black} o_{2}$};
\node[draw, circle, minimum size=\radius] (n10) at (1,1) {$\color{blue!90!black} o_{3}$};
\node[draw, circle, minimum size=\radius] (n11) at (2,1) {$\color{green!70!black} s_{1}$};
\node[draw, circle, minimum size=\radius] (n12) at (3,1) {$\color{green!70!black} s_{2}$};
\node[draw, circle, minimum size=\radius] (n13) at (4,1) {};
\node[draw, circle, minimum size=\radius] (n14) at (5,1) {};
\node[draw, circle, minimum size=\radius] (n15) at (6,1) {};
\node[draw, circle, minimum size=\radius] (n16) at (7,1) {};
\node[draw, circle, red, minimum size=\radius] (n17) at (0,2) {$i$};
\node[draw, circle, minimum size=\radius] (n18) at (1,2) {$\color{blue!90!black} o_{1}$};
\node[draw, circle, minimum size=\radius] (n19) at (2,2) {$\color{green!70!black} s_{3}$};
\node[draw, circle, minimum size=\radius] (n20) at (3,2) {$\color{green!70!black} s_{4}$};
\node[draw, circle, minimum size=\radius] (n21) at (4,2) {};
\node[draw, circle, minimum size=\radius] (n22) at (5,2) {};
\node[draw, circle, minimum size=\radius] (n23) at (6,2) {};
\node[draw, circle, minimum size=\radius] (n24) at (7,2) {};
	
\path[-,line width=0.20mm]
	(n1) edge (n2)
	(n2) edge (n3)
	(n3) edge (n4)
	(n4) edge (n5)
	(n5) edge (n6)
	(n6) edge (n7)
	(n7) edge (n8)
	(n9) edge (n10)
	(n10) edge (n11)
	(n11) edge (n12)
	(n12) edge (n13)
	(n13) edge (n14)
	(n14) edge (n15)
	(n15) edge (n16)
	(n17) edge (n18)
	(n18) edge (n19)
	(n19) edge (n20)
	(n20) edge (n21)
	(n21) edge (n22)
	(n22) edge (n23)
	(n23) edge (n24)
	(n1) edge (n9)
	(n2) edge (n10)
	(n3) edge (n12)
	(n5) edge (n13)
	(n6) edge (n14)
	(n7) edge (n15)
	(n8) edge (n16)
	(n9) edge (n17)
	(n10) edge (n18)
	(n11) edge (n19)
	(n12) edge (n20)
	(n13) edge (n21)
	(n14) edge (n22)
	(n15) edge (n23)
	(n16) edge (n24)
	(n1) edge (n10)
	(n2) edge (n11)
	(n4) edge (n13)
	(n5) edge (n14)
	(n6) edge (n15)
	(n7) edge (n16)
	(n9) edge (n18)
	(n10) edge (n19)
	(n11) edge (n20)
	(n12) edge (n21)
	(n13) edge (n22)
	(n14) edge (n23)
	(n15) edge (n24)
	(n2) edge (n9)
	(n3) edge (n10)
	(n3) edge (n11)
	(n5) edge (n12)
	(n6) edge (n13)
	(n7) edge (n14)
	(n8) edge (n15)
	(n10) edge (n17)
	(n11) edge (n18)
	(n12) edge (n19)
	(n13) edge (n20)
	(n14) edge (n21)
	(n15) edge (n22)
	(n16) edge (n23)
;

\path[-,line width=0.20mm]
	(n4) edge (n11)
	(n4) edge (n12)
;

\path[-,line width=0.40mm, magenta]
	(n4) edge (n11)
	(n4) edge (n12)
;

\draw[line width=0.4mm, green!70!black] (-0.5,0.5) rectangle (3.5,2.5);

\end{tikzpicture}
\captionof{figure}{Example of node $i$ being $2$-information robust}
\label{fig:i_is_r_info_robust}
\end{minipage}
\vspace{0mm}

The following lemma shows that if node $i$ is $r$-information robust then if at each node $r-1$ in-edges are removed the resulting graph is rooted at $i$. 

\begin{lemma} \label{lemma:rooted_edge_removed}
	Given a communication graph $\mathcal{G}_{c} = (\mathcal{N},\mathcal{E}_{c})$ and an observation graph $\mathcal{G}_{o} = (\mathcal{N},\mathcal{E}_{o})$, assume that node $i$ is $r$-information robust. Let $\tilde{\mathcal{E}}_{c}$ be the set of edges after $(r-1)$ in-edges from each node $j\in\mathcal{N}$ have been removed from $\mathcal{E}_{c}$. Then, the graph $\mathcal{G} = (\mathcal{N},\tilde{\mathcal{E}}_{c}\cup \mathcal{E}_{o})$ is rooted at $i$. \hfill \QEDopen
\end{lemma}
\begin{proof}
Since $\mathcal{E}_{o}$ is contained in the edge set of $\mathcal{G}$, there is an edge from node $i$ to the nodes $\mathcal{N}^{out}_{i}(\mathcal{G}_{o})$. Let $X(i)$ denote all nodes in $\mathcal{G}$ that have a path from $\mathcal{N}^{out}_{i}(\mathcal{G}_{o})$ to them. Therefore, $Z(i) = \mathcal{N}^{out}_{i}(\mathcal{G}_{o})\cup X(i)$ are all nodes that have a path from $i$. 
	
\begin{minipage}[ht]{1\columnwidth}
\center
\begin{tikzpicture}
\def\Nnodes{3}
\def\radius{15mm}
\def\ang{30}

\node[circle,fill=black,inner sep=0pt,minimum size=3pt,label=below:{$i$}] (r) at (0,-0.5) {};
\node[ellipse, minimum width=2mm, minimum height=20mm, draw] (I) at (0,0) {$\mathcal{N}^{out}_{i}(\mathcal{G}_{o})$};
\node[ellipse, minimum width=2mm, minimum height=20mm, draw] (Y) at (2.60,0) {$X(i)$};
\node[draw, ellipse, minimum width=50mm, minimum height=30mm] (Z) at (1.1,0) {};
\node (z) at (1.4,1.1) {$Z(i)$};

\node[circle,fill=black,inner sep=0pt,minimum size=3pt,label=right:{$j$}] (p) at (4.9,0.9) {};
\node[draw, ellipse, minimum width=5mm, minimum height=30mm] (W) at (5,0) {$\bar{Z}(i)$};
	
\path[->,line width=0.20mm]
	(I) edge[bend right=\ang] (Y)
	(I) edge[bend right=-\ang] (Y)
	(Z) edge[bend right=\ang, dashed] (W)
	(Z) edge[bend right=-\ang, dashed] (W)
	(Z) edge[bend right=15, dashed] (W)
	(Z) edge[bend right=-15, dashed] (W)
	(Z) edge[bend right=-45] (p)
;

\end{tikzpicture}

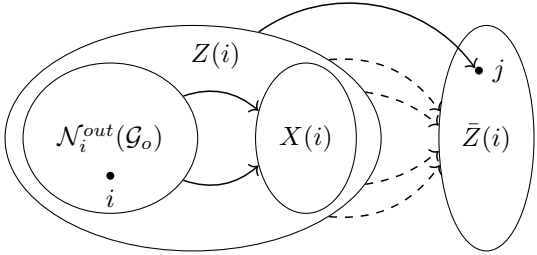
\captionof{figure}{State of node $i$ throughout network}
\label{fig:info_spread}
\end{minipage}
\vspace{2mm}
	
	We will show that $Z(i) = \mathcal{N}$. Assume that the set $\bar{Z}(i)$ is not empty and is depicted in Figure \ref{fig:info_spread}. Note that the set $Z(i) \supset \mathcal{N}^{out}_{i}(\mathcal{G}_{o})$ and by the assumption that node $i$ is $r$-information robust there must be a node $j\in \bar{Z}(i)$ that has at least $r$ edges in $\mathcal{E}_{c}$ connecting from $Z(i)$ to $\bar{Z}(i)$. After removing $(r-1)$ in-edges from node $j$ there must be at least one edge going from $Z(i)$ to $\bar{Z}(i)$. Therefore, node $j$ should have been in $Z(i)$ since there is a path from $i$ to $j$. Repeating this process will remove all nodes from $\bar{Z}(i)$. Therefore, our assumption that $\bar{Z}(i)$ is not empty is false, $Z(i) = \mathcal{N}$, and there is a path from $i$ to any node on the graph, i.e., $\mathcal{G}$ is rooted at $i$.
\end{proof}

\section{Proposed Algorithm} \label{sec:algo}

In this section we present our proposed algorithm for NE seeking with truthful and adversarial agents. We make the following assumption about the number of adversarial agents that communicate with truthful agents.
\begin{asmp} \label{asmp:d_local}
	For the communication graph $\mathcal{G}_{c}$ the set $\mathcal{A}$ is a $D$-local set and $(\forall i \in\mathcal{T})$, $|\mathcal{N}^{in}_{i}(\mathcal{G}_{c})| \geq 2D+1$. \hfill \QEDopen
\end{asmp}
	This assumption states that around every truthful agent there should be more truthful agents than adversarial ones, i.e., that a truthful agent $i$ receives communicated messages from at most $D$ adversarial agents and at least $D+1$ truthful neighbours.

 The algorithm is based on the idea of using some intermediary estimate denoted by $\mathbf{v}^{i}$. In constructing $\mathbf{v}^{i}$, agent $i$ uses the actions that can be directly observed via $\mathcal{G}_{o}$ into the estimate, while relying on communicated messages from neighbours for those actions that cannot be directly observed via $\mathcal{G}_{c}$. In constructing $\mathbf{v}^{i}$ for actions that are not directly observed, agent $i$ removes/prunes extreme data values received from neighbours and takes a weighted average on the remaining ones. After the estimate $\mathbf{v}^{i}$ is constructed it is used by agent $i$ in a gradient step to update their own action towards minimizing their cost.

\newpage

\noindent\rule[0.5ex]{\linewidth}{1pt}
\vspace{-7.5mm}

\noindent\rule[0.5ex]{\linewidth}{1pt}
\textbf{Algorithm 1:} Robust NE Seeking Algorithm 

\noindent\rule[0.5ex]{\linewidth}{1pt}
\vspace{-7mm}
\begin{Numera}
\num At any iteration $k$, each agent $i\in\mathcal{N}$ sends $\mathbf{y}^{ji}[k]$ to all $ j\!\in\!\mathcal{N}_{i}^{out}(\mathcal{G}_{c})$, and receives $\mathbf{y}^{ij}[k]$ from all $j\!\in \!\mathcal{N}_{i}^{in}(\mathcal{G}_{c})$. 

\num For each $m\in\mathcal{N}$ and $q\in\set{1,2,\dots,n_{m}}$, agent $i\in\mathcal{T}$ removes the $D$ highest and $D$ smallest values $\mathbf{y}^{ij}_{mq}$ that are larger and smaller than its own value $\mathbf{x}^{i}_{mq}$, respectively (ties broken arbitrarily). If there are fewer than $D$ values higher (resp. lower) than its own value, $i$ removes all of those values. Let $\mathcal{Y}^{i}_{mq}$ be the set of agents whose message about the $q^{th}$ component of $m$'s action is retained by $i$ and agent $i$ itself. Then, $\forall q\in\set{1,2,\dots, n_{m}}$

\begin{align} \label{eqn:comm_step}
\mathbf{v}^{i}_{mq}[k] &= \begin{cases}
	\underset{j\in\mathcal{Y}^{i}_{mq}}{\sum}  w^{ij}_{mq}[k]\mathbf{y}^{ij}_{mq}[k] & \text{if } m\not\in \mathcal{N}_{i}^{in}(\mathcal{G}_{o}) \\
	x_{mq}[k] & \text{else }
\end{cases}
\end{align}
where the weights $w^{ij}_{mq}$ are such that for all $i,j\in\mathcal{N}$, $\sum_{m\in\mathcal{N}} w^{ij}_{mq}[k] = 1$; for all $i\in\mathcal{N}$ and $j \not\in \mathcal{Y}^{i}_{mq}$, $w^{ij}_{mq}[k] = 0$; and there $\exists \eta > 0$, $\forall i\in\mathcal{N}$ and $j \in \mathcal{Y}^{i}_{mq}$, $w^{ij}_{mq}[k] \geq \eta$.

\num Agent $i\in\mathcal{T}$ updates its estimate vector as
\begin{align*}
	\mathbf{x}^{i}[k+1] &= \mathbf{v}^{i}[k] - \alpha \mathcal{R}^{T}_{i}\frac{\partial J_{i}(\mathbf{v}^{i}[k])}{\partial x_{i}}
\end{align*}
where $\mathcal{R}_{i} = [\mathbf{0}_{n_{i}\times n^{<}_{i}}\ I_{n_{i}}\ \mathbf{0}_{n_{i} \times n^{>}_{i}}]$, $n^{<}_{i} = \sum_{\substack{j<i}}n_{j}$, $n^{>}_{i} = \sum_{\substack{j>i}}n_{j}$, $\mathbf{v}^{i}_{m} = \col(\mathbf{v}^{i}_{mq})_{q\in\set{1,2,\dots,n_{m}}}$, and $\mathbf{v}^{i} = \col(\mathbf{v}^{i}_{m})_{m\in\mathcal{N}}$.
\num Agent $i\in\mathcal{T}$ updates its communicated message $\mathbf{y}^{ji}[k+1]=\mathbf{x}^{i}[k+1]$ 
\end{Numera}
\vspace{-5mm}
\noindent\rule[0.5ex]{\linewidth}{1pt}
\vspace{-7.5mm}

\noindent\rule[0.5ex]{\linewidth}{1pt}

In the above, $\mathbf{v}^{i}_{mq}$ is the intermediate estimate of agent $i$ on the $q^{th}$ component of agent $m$'s action, while $w^{ij}_{mq}\in [0,1]$ is the weight that agent $i$ places on information from agent $j$ about the $q^{th}$ component of agent $m$'s action. Step 2 is a filtering process, where for each message $\mathbf{y}^{ij}_{mq}$ removed, the corresponding edge in $ \mathcal{G}_{c}$ as used by agent $i$ is removed (set $w^{ij}_{mq} = 0$). Notice that the filtering process results in an algorithm with a switching/time-varying communication graph.

Step 2 of the algorithm is inspired by \cite{robust_opt_ghar}, which used this type of pruning for a distributed optimization problem. The major difference between the pruning step \cite{robust_opt_ghar} and ours, is the use of the observation graph that allows us to obtain stronger convergence results. In \cite{robust_opt_ghar}, the authors only focused on finding the minimizer of the set of truthful agents in the network and ignored the actions of the adversarial agents. For the Nash equilibrium problem, the agents' cost functions are dependent on the actions of the adversarial agents and therefore can not be ignored. We seek to find the NE, $x^{*}$, which includes both truthful and adversarial agent's actions. In \cite{robust_opt_ghar}, a solution in the convex hull of the local minimizers is found but this does not guarantee to be a global minimizer. Herein, we wish to converge to the true NE and not just within a neighborhood of the NE. Lastly, \cite{robust_opt_ghar} ignores the dynamics of the adversarial agents. Herein, every $j\in\mathcal{A}$ has a optimization problem it wants to solve, which is coupled to the actions/estimates of the others. Therefore, we also need to model how adversarial agents update their estimates/actions. 

 Since, adversarial agents are motivated to solve \eqref{ne-pi-a}, use their messages to deceive other agents, and to avoid being deceived, we make the following assumption that characterizes this behavior.

\begin{asmp} \label{asmp:adversary_dynamics}
For each $i\in \mathcal{A}$, agent $i$ can update their estimate $\mathbf{v}^{i}$ as a weighted average of only truthful agents messages and agent $i$ updates $\mathbf{x}^{i}$ via step 3 of Algorithm 1. \hfill \QEDopen
\end{asmp}
\begin{remark}
	We justify Assumption \ref{asmp:adversary_dynamics} by first noting that  each agent $i\in \mathcal{A}$ is solving an optimization problem and using gradient descent type dynamics is a common method used. 
	The assumption that agent $i$ updates ``$\mathbf{v}^{i}$ as a function of just truthful agents messages'' obviously holds when $|\mathcal{A}|=1$. When $|\mathcal{A}| > 1$, we argue that each agent $i\in\mathcal{A}$ is trying to solve its own optimization problem, and therefore needs to know what the actions of the other agents are. It only seems reasonable that if agent $i$ is trying to deceive the other agents, that agent $i$ would also protect itself from being deceived by others.	Here are some instances where this assumption holds: 
\begin{itemize}
\item If a set of agents create a coalition, $C$ and spread misinformation to agents in $\bar{C}$. Then, all the agents in $C$ would know which information can be trusted.
\item If the adversarial agents are spread out and do not directly communicate to each other.
\item If adversarial agents are able to get access to all the actions $x_{j}$ directly. 
\item If agent $i$'s cost function is $J_{i}(x) = 0$ and only wants to disrupt the system.
\item If agent $i$ uses Step 2 of Algorithm 1 to filter out other adversarial agents as a method to protect itself.
\end{itemize}
\end{remark}

\section{Convergence Analysis} \label{sec:conv}

To prove that Algorithm 1 converges to the NE, we first show that Step 2 of the algorithm can be equivalently written in terms of truthful agents only. 

\begin{lemma} \label{lemma:only_truth_equiv}
	Consider the observation graph $\mathcal{G}_{o}$, and communication $\mathcal{G}_{c}$ under Assumption \ref{asmp:d_local}. Then, $(\forall m\in \mathcal{N})$ $(\forall q\in \set{1,2,\dots,n_{m}})$ $(\forall i \in\mathcal{T})$, \eqref{eqn:comm_step} can be written as,
\begin{align} \label{eqn:equiv_comm_step}
\mathbf{v}^{i}_{mq}[k] &= \begin{cases}
	\underset{j\in \tilde{\mathcal{Y}}^{i}_{mq}}{\sum}  \tilde{w}^{ij}_{mq}[k]\mathbf{y}^{ij}_{mq}[k] & \text{if } m\not\in \mathcal{N}_{i}^{in}(\mathcal{G}_{o}) \\
	x_{mq}[k] & \text{else }
\end{cases}
\end{align}
where $\tilde{w}^{ij}_{mq}$ are new weights and $\tilde{\mathcal{Y}}^{i}_{mq} \subset \bracket{\mathcal{N}^{in}_{i}(\mathcal{G}_{c})\cap \mathcal{T}}\cup \set{i}$ are truthful agents whose message is retained. The weights $\tilde{w}^{ij}_{mq}$ are such that for all $i,j\in\mathcal{N}$, $\sum_{m\in\mathcal{N}} \tilde{w}^{ij}_{mq}[k] = 1$; for all $i\in\mathcal{N}$ and $j \not\in \tilde{\mathcal{Y}}^{i}_{mq}$, $\tilde{w}^{ij}_{mq}[k] = 0$; $\forall i\in\mathcal{N}$ and $j \in \tilde{\mathcal{Y}}^{i}_{mq}$, $\tilde{w}^{ij}_{mq}[k] > 0$; and $\forall i\in\mathcal{N}$ there is at least $|\mathcal{N}^{in}_{i}(\mathcal{G}_{c})| - 2D$ nodes $j\in \tilde{\mathcal{Y}}^{i}_{mq}$ with $\tilde{w}^{ij}_{mq} \geq \frac{\eta}{2}$. \hfill \QEDopen
\end{lemma}
\begin{proof}
Note that the difference between \eqref{eqn:comm_step} and \eqref{eqn:equiv_comm_step} is that the set $\mathcal{Y}^{i}_{mq}$ and the weights $w^{ij}_{mq}$ are replaced with $\tilde{\mathcal{Y}}^{i}_{mq}$ and $\tilde{w}^{ij}_{mq}$ respectively. The set $\mathcal{Y}^{i}_{mq}$ is the set of agents whose message is retained, which can contain adversarial agents, while the set $\tilde{\mathcal{Y}}^{i}_{mq}$ is the set of \emph{only truthful} agents retained.

Since $\mathbf{v}^{i}_{mq}$ is the same for both \eqref{eqn:comm_step} and \eqref{eqn:equiv_comm_step} when $m \in \mathcal{N}^{in}_{i}(\mathcal{G}_{o})$, we only need to consider $m\not\in \mathcal{N}^{in}_{i}(\mathcal{G}_{o})$, i.e., agents unobserved by $i$. The rest of the proof follows the same logic as Proposition 5.1 \cite{robust_opt_ghar}.  The idea of the proof is that either all adversaries are filtered out,  ($\mathcal{Y}^{i}_{mq} = \tilde{\mathcal{Y}}^{i}_{mq}$), or messages from adversaries that are not filtered out can be written as a linear combination of agents in $\mathcal{T}$. This is because if $\mathbf{y}^{ij}_{mq}$ for $j \in\mathcal{A}$ is not filtered out, then there exists $s,t\in\mathcal{T}$ such that $\mathbf{y}^{is}_{mq} \leq \mathbf{y}^{ij}_{mq} \leq \mathbf{y}^{it}_{mq}$ and $\mathbf{y}^{ij}_{mq} = (1-\alpha)\mathbf{y}^{is}_{mq} + \alpha \mathbf{y}^{it}_{mq}$ for some $\alpha \in [0,1]$. The weight $w^{ij}_{mq}$ of the adversarial node would be added to the weights of $w^{is}_{mq}$ and $w^{it}_{mq}$ by $(1-\alpha)w^{ij}_{mq}$ and $\alpha w^{ij}_{mq}$, respectively, resulting in  $\tilde{w}_{mq}^{is}$ and $\tilde{w}_{mq}^{it}$.
\end{proof}

Notice that under Assumption \ref{asmp:adversary_dynamics}, adversarial agent $i\in\mathcal{A}$ can update $\mathbf{v}^{i}_{mq}$ as a function of just truthful messages. Similarly, by Lemma \ref{lemma:only_truth_equiv} truthful agents $i\in\mathcal{T}$ can also update $\mathbf{v}^{i}_{mq}$ as a function of just truthful messages. Additionally, since truthful messages must satisfy \eqref{eqn:truth_comm} we can express $\mathbf{v}^{i}_{mq}$ as a function of $\mathbf{x}^{j}_{mq}$. For agent $i\in\mathcal{T}$ let $\tilde{w}^{i}_{mq} = \row(\tilde{w}^{ij}_{mq})_{j\in\mathcal{N}}$ where $\tilde{w}^{ij}_{mq}$ are the weights from Lemma \ref{lemma:only_truth_equiv}, and let $\mathbf{x}_{mq} = \col(\mathbf{x}^{j}_{mq})_{j\in\mathcal{N}}$. Then, we can compactly write $\mathbf{v}^{i}_{mq} = \tilde{w}^{i}_{mq}\mathbf{x}_{mq}$ when $m\not\in \mathcal{N}_{i}^{in}(\mathcal{G}_{o})$. With abuse of notation for adversarial agent $i\in\mathcal{A}$, let $\tilde{w}^{ij}_{mq}$ denote the weights that adversarial agents places on the messages received. Then, adversarial agents can also write $\mathbf{v}^{i}_{mq} = \tilde{w}^{i}_{mq}\mathbf{x}_{mq}$ when $m\not\in \mathcal{N}_{i}^{in}(\mathcal{G}_{o})$. Let $E_{mq} \in\reals^{n\times n}$ be a diagonal matrix with the $(n^{<}_{m} + q)$-diagonal entry equal to $1$ and $0$'s elsewhere, $\tilde{W}^{i} = \sum_{m\in\mathcal{N}}\sum_{q\in\set{1,2,\dots,n_{m}}} \tilde{w}^{i}_{mq}\otimes E_{mq}$, and $O^{i}\in\reals^{n\times n}$ be a diagonal matrix where the $(n^{<}_{m} + q)$-diagonal entry is $1$ if agent $i$ can observe $m$ and $0$ otherwise. Then, for agent $i$ we can compactly write
\begin{align}
\mathbf{v}^{i} = \bracket{I - O^{i}}\tilde{W}^{i}\mathbf{x} + O^{i}\mathbf{x} \label{eqn:compact_i}
\end{align}
The following lemma gives a compact form for Algorithm 1.
\begin{lemma} \label{lemma:compact_dynamics}
	Under Assumption \ref{asmp:d_local} and \ref{asmp:adversary_dynamics}, Algorithm 1 is equivalent to,
\begin{align}
	\mathbf{v}[k] &= \overline{W}[k]\mathbf{x}[k] \label{eqn:compact_W_truthful} \\
	\mathbf{x}[k+1] &= \mathbf{v}[k] - \alpha \mathcal{R}^{T}\mathbf{F}(\mathbf{v}[k]) \label{eqn:compact_gradient}
\end{align}
where $\mathcal{R} = diag(\mathcal{R}_{i})_{i\in\mathcal{N}}$, $\mathbf{F}(\mathbf{x}) = \col\bracket{\frac{\partial J_{i}(\mathbf{x}^{i})}{\partial x_{i}}}_{i\in\mathcal{N}}$ is the \emph{extended pseudo-gradient}, $\mathbf{v} = \col(\mathbf{v}^{i})_{i\in\mathcal{N}}$, and
\begin{align}
	\overline{W}[k] &= \bracket{I - O}\tilde{W}[k] + O \label{eqn:Wk}
\end{align}
where $O = \col(O^{i})_{i\in\mathcal{N}}$, and $\tilde{W}[k] = \col(\tilde{W}^{i}[k])_{i\in\mathcal{N}}$. \hfill \QEDopen
\end{lemma}
\begin{proof}
Trivially obtained by stacking equation \eqref{eqn:compact_i} and Step $3$ of algorithm 1. Note that Step 3 also applies to adversarial agents by Assumption \ref{asmp:adversary_dynamics}.
\end{proof}

While Lemma \ref{lemma:only_truth_equiv} does show that step 2 of the algorithm can be described using messages sent only by truthful agents, this is not enough to ensure that agents learn $x$. This is because we have not assumed anything on the structure of $\overline{W}[k]$. Figure \ref{fig:equiv_weight} shows a counter example, where $\mathcal{A}=\set{4}$ is $1$-local, $|\mathcal{N}^{in}_{i}(\mathcal{G}_{c})| \geq 3$, and satisfies Assumption \ref{asmp:d_local}.  However, it is obvious that agents will not learn $x$.

\vspace{2mm}
\begin{minipage}[ht]{1\columnwidth}
\center
\begin{tikzpicture}
\def\Nnodes{3}
\def\radius{15mm}
\def\T{green!65!black}
\def\A{red!90!black}

\node[draw, circle, \T] (n1) at (0,0) {1};
\node[draw, circle, \T] (n2) at (1,1) {2};
\node[draw, circle, \T] (n3) at (1,-1) {3};
\node[draw, circle, \A] (n4) at (2,0) {4};
\node[draw, circle, \T] (n5) at (3,1) {5};
\node[draw, circle, \T] (n6) at (3,-1) {6};
\node[draw, circle, \T] (n7) at (4,0) {7};

\path[<->,line width=0.20mm]
	(n1) edge (n2)
	(n1) edge (n3)
	(n1) edge (n4)
	(n2) edge (n4)
	(n2) edge (n3)
	(n3) edge (n4)
	(n4) edge (n5)
	(n4) edge (n6)
	(n4) edge (n7)
	(n5) edge (n7)
	(n5) edge (n6)
	(n6) edge (n7)
;

\end{tikzpicture}

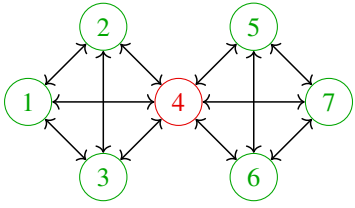
\captionof{figure}{Graph of $\mathcal{G}_{c} = \mathcal{G}_{o}$}
\label{fig:equiv_weight}
\end{minipage}
\vspace{0mm}

From Figure \ref{fig:equiv_weight} we see that information about $x_{7}$ will never reach nodes $1,2,3$. Node $4$ can always manipulate the information about $x_{7}$ before it reaches node $1,2,3$. We can see that the graph is not connected enough for information to flow throughout the network. To deal with this case we make the following assumption.

\begin{asmp} \label{asmp:info_robust}
Given the communication graph $\mathcal{G}_{c}$ and observation graph $\mathcal{G}_{o}$, all nodes $m\in\mathcal{N}$ are $(2D+1)$-information robust. \hfill \QEDopen
\end{asmp}

\begin{remark}
This assumption ensures that after the pruning process of Step 2 in Algorithm 1, the information about any agent $m$ can reach all agents in $\mathcal{N}$.
\end{remark}

\begin{lemma} \label{lemma:equiv_rooted}
	Consider the communication graph $\mathcal{G}_{c}$ and observation graph $\mathcal{G}_{o}$, under Assumption \ref{asmp:d_local}, \ref{asmp:adversary_dynamics}, and \ref{asmp:info_robust}. Then, for each $m \in\mathcal{N}$ and $q\in\set{1,2,\dots,n_{m}}$, the graph representing the information exchange of the $q^{th}$ component of agent $m$'s action induced by \eqref{eqn:Wk} has a path from $m$ to every other node, i.e., $m$ is rooted. \hfill \QEDopen
\end{lemma}
\begin{proof}
For each $m \in\mathcal{N}$ and $q\in\set{1,2,\dots,n_{m}}$, we know from Lemma \ref{lemma:only_truth_equiv} that each agent in $\mathcal{T}$ has $|\tilde{\mathcal{Y}}^{i}_{mq}|\geq |\mathcal{N}^{in}_{i}(\mathcal{G}_{c})|-2D$. Therefore, each agent has removed at most $2D$ nodes. Since, the graph is $(2D+1)$-information robust, from Lemma \ref{lemma:rooted_edge_removed} we know that there is at least one path from $m$ to every node in $\mathcal{T}$. Since there is a path to each truthful agent and adversarial agents get their messages from only truthful agents by assumption \ref{asmp:adversary_dynamics} we know that their is a path to all adversarial agents as well.
\end{proof}

\begin{remark}
	In the literature for resilient consensus algorithms \cite[Thm. 2]{robust_consensus_0} and resilent distributed optimization assumption \cite[Thm 6.1]{robust_opt_ghar}, Assumption \ref{asmp:d_local} is a standard assumption. If Assumption \ref{asmp:d_local} does not hold then adversarial agents can not be filtered out and consensus can not occur. Additionally, the assumption that the graph is $(2D+1)$-robust \cite{robust_consensus_0}, which is very closely related to $(2D+1)$-information robust (Assumption \ref{asmp:info_robust}), is a standard assumption that ensure that the graph is rooted after removing edges \cite[Thm. 2]{robust_consensus_0} \cite[Thm 6.1]{robust_opt_ghar}.
\end{remark}

Before proving that our proposed algorithm converges to the NE, we first highlight some important properties $\overline{W}[k]$ has under Assumption \ref{asmp:adversary_dynamics} and \ref{asmp:info_robust}. The matrix $\overline{W}[k]$ is a time-varying row stochastic matrix with $\overline{W}[k]\bracket{\mathbf{1}_{N}\otimes v} = \mathbf{1}_{N}\otimes v$, $\forall k\geq 0$ and $\forall v\in \reals^{n}$. Additionally, $\mathcal{R}\overline{W}[k] = \mathcal{R}$, $\forall k\geq 0$. Let $A \defeq \frac{1}{\sqrt{N}} \mathbf{1}_{N}\otimes\mathcal{R}\in\reals^{Nn\times Nn}$ and notice that $A\overline{W}[k] = \overline{W}[k]A = A$, $A^2 = A$, and $\norm{A} = 1$, i.e., $A$ is a projection matrix and it commutes with $\overline{W}[k]$. The following lemma gives an important property of the matrix  $\overline{W}[k]$.
\begin{lemma} \label{lemma:lyapunov_ltv}
	Consider the communication graph $\mathcal{G}_{c}$ and the observation graph $\mathcal{G}_{o}$, under Assumptions \ref{asmp:d_local}, \ref{asmp:adversary_dynamics}, and \ref{asmp:info_robust}. If $\overline{W}[k]$ is as in Lemma \ref{lemma:compact_dynamics} then there exists a sequence of positive semi-definite matrices $P[k]$ such that,
\begin{align*}
	\overline{W}[k]^{T}P[k+1]\overline{W}[k] = P[k]- (I-A)^{T}(I-A)
\end{align*}
where $I$ is the identity matrix of appropriate dimensions, and $Null(P[k]) = \mathbf{1}\otimes v$, for all $k\geq 0$, $v \in \reals^n$. 
Additionally, 
\begin{align*}
	P[k] \preceq \overline{p}I 
\end{align*}
where $ \overline{p} = \frac{4N}{1 - C^{\frac{2}{N-1}}}$, with $C =1-\bracket{\frac{\eta}{2}}^{N-1}$. \hfill \QEDopen
\end{lemma}
\begin{proof}
	Found in the Appendix
\end{proof}

\begin{remark} \label{remark:conservative}
	The upper bound $\overline{p}$ is very conservative because it is obtained assuming the worst possible communication graph switching. In practice $\bar{p}$ is significantly lower.
\end{remark}

\begin{thm} \label{thm:robust_gradient}
	Given a game $G$, with communication graph $\mathcal{G}_{c}$ and observation graph $\mathcal{G}_{o}$ such that Assumptions 1 - 4 hold, let
\begin{align}\label{eq_M}
	M \defeq \begin{bmatrix}
		2\alpha \mu - \overline{p}\alpha^{2}L^{2} & -\alpha \overline{p}L\bracket{1+\alpha L} \\
		-\alpha \overline{p}L\bracket{1+\alpha L} & 1 - 2\alpha (\overline{p}-1)L - \overline{p}\alpha^{2}L^{2}
	\end{bmatrix}
\end{align}
with $\overline{p}$ as in Lemma \ref{lemma:lyapunov_ltv}, and select $\alpha$ such that $M\succ \mathbf{0}$.
Then for any initial condition $\mathbf{x}_{0}, \mathbf{v}_{0}$, the iterates generated by Algorithm 1 $\col(\mathbf{x}^{i}[k])_{i\in\mathcal{N}} \to \mathbf{1}\otimes x^{*}$, where $x^{*}$ is the Nash equilibrium. \hfill \QEDopen
\end{thm}
\begin{proof}
By Lemma \ref{lemma:compact_dynamics} the updates of Algorithm 1 are compactly written as \eqref{eqn:compact_W_truthful} and \eqref{eqn:compact_gradient}. To simplify notation, in the following we denote iteration $[k]$ with the subscript $k$ and $H \defeq \mathcal{R}^{T}\mathbf{F}$. Then, from \eqref{eqn:compact_W_truthful} and  \eqref{eqn:compact_gradient},   $\mathbf{v}_{k+1}=\overline{W}_{k+1}(\mathbf{v}_{k}-\alpha H\mathbf{v}_{k} )$. Consider the following candidate Lyapunov function,
\begin{align}
	V_{k+1} = \snorm{\mathbf{v}_{k+1}-A\mathbf{v}_{k+1}}_{P_{k+1}} + \snorm{A\mathbf{v}_{k+1} - \mathbf{x}^{*}} \label{eqn:lyap}
\end{align}
where  $P_{k+1}$ is as in Lemma \ref{lemma:lyapunov_ltv} and $\mathbf{x}^{* } = \mathbf{1}\otimes x^{*}$. It can be shown that  $V_{k+1}$ is positive definite at $\mathbf{x}^{* }$ and radially unbounded. Indeed, the first term in $V_{k+1}$ is zero only when $\mathbf{v}_{k+1}=A\mathbf{v}_{k+1}$, because the null space  of $P_{k+1}$ is the null space of $(I-A)$ for all $k$ (Lemma \ref{lemma:lyapunov_ltv}), i.e., when $\mathbf{v}_{k+1}$ is equal to its projection onto the consensus subspace. Thus, $V_{k+1}$ is zero only when also $\snorm{A\mathbf{v}_{k+1} - \mathbf{x}^{*}} = \snorm{\mathbf{v}_{k+1} - \mathbf{x}^{*}}$, which implies $\mathbf{v}_{k+1} = \mathbf{x}^{*}$. 

Note that $A\overline{W}_{k}=\overline{W}_{k}A = A$ for all $k$ and $\norm{I-A} = 1$. Additionally, from Lemma \ref{lemma:lyapunov_ltv} we know that $P_{k} \preceq \overline{p} I$. Then, from \eqref{eqn:compact_W_truthful} and \eqref{eqn:compact_gradient}, the first term in \eqref{eqn:lyap} is equal to
\begin{align*}
&\snorm{\mathbf{v}_{k+1}-A\mathbf{v}_{k+1}}_{P_{k+1}} \\
&= \snorm{(I-A)\overline{W}_{k}\mathbf{x}_{k+1}}_{P_{k+1}} = \snorm{\overline{W}_{k}(I-A)\mathbf{x}_{k+1}}_{P_{k+1}} \\
&= \snorm{(I-A)\mathbf{x}_{k+1}}_{P_{k}-I} = \snorm{(I-A)(I-\alpha H)\mathbf{v}_{k}}_{P_{k}-I} \\
&= \snorm{(I-A)\mathbf{v}_{k} - \alpha (I-A)H\mathbf{v}_{k}}_{P_{k}-I} \\
&= ||(I-A)\mathbf{v}_{k} \\
&\quad - \alpha (I-A)\bracket{H\mathbf{v}_{k} - HA\mathbf{v}_{k} + HA\mathbf{v}_{k} - H\mathbf{x}^{*}}||^{2}_{P_{k}-I}
\end{align*}
where we used that $A\overline{W}_{k} = \overline{W}_{k}A$, and Lemma \ref{lemma:lyapunov_ltv} on line 3. Additionally, we used Remark \ref{remark:ne_partial_char} ($\mathbf{F}(\mathbf{x}^{*}) = \mathbf{0}$) to add $H\mathbf{x}^{*} = \mathcal{R}^{T} \mathbf{F}(\mathbf{x}^{*}) = \mathbf{0}$. Expanding out the right hand side we get that,
\begin{align*}
&\snorm{\mathbf{v}_{k+1}-A\mathbf{v}_{k+1}}_{P_{k+1}} \\
&= \snorm{(I-A)\mathbf{v}_{k}}_{P_{k}-I} + \alpha^{2} \snorm{(I-A)\bracket{H\mathbf{v}_{k}-HA\mathbf{v}_{k}}}_{P_{k}-I} \\
	&\quad + \alpha^{2} \snorm{(I-A)\bracket{HA\mathbf{v}_{k}-H\mathbf{x}^{*}}}_{P_{k}-I} \\
	&\quad -2\alpha \inp*{(I-A)\mathbf{v}_{k}}{(I-A)\bracket{H\mathbf{v}_{k}-HA\mathbf{v}_{k}}}_{P_{k}-I} \\
	&\quad -2\alpha \inp*{(I-A)\mathbf{v}_{k}}{(I-A)\bracket{HA\mathbf{v}_{k}-H\mathbf{x}^{*}}}_{P_{k}-I} \\
	+& 2\alpha^{2} \inp*{(I\! - \! A) \! \bracket{H\mathbf{v}_{k}-HA\mathbf{v}_{k}}\! }{\! (I \! - \! A)\! \bracket{HA\mathbf{v}_{k}-H\mathbf{x}^{*}}}_{P_{k}-I}
\end{align*}
Next we use the fact that $\snorm{a}_{P_{k}-I} \leq (\overline{p}-1)\snorm{a}$ and $\inp*{a}{b}_{P_{k}-I} \leq (\overline{p}-1)\norm{a}\norm{b}$. Thus,
\begin{align*}
&\snorm{\mathbf{v}_{k+1}-A\mathbf{v}_{k+1}}_{P_{k+1}} \\
&\quad \leq \snorm{\mathbf{v}_{k}-A\mathbf{v}_{k}}_{P_{k}} - \snorm{\mathbf{v}_{k}-A\mathbf{v}_{k}} \\
&\quad + \alpha^{2}(\overline{p}-1)L^{2} \snorm{\mathbf{v}_{k}-A\mathbf{v}_{k}} \\
	&\quad + \alpha^{2}(\overline{p}-1)L^{2} \snorm{A\mathbf{v}_{k}-\mathbf{x}^{*}} \\
	&\quad +2\alpha (\overline{p}-1)L \snorm{\mathbf{v}_{k}-A\mathbf{v}_{k}} \\
	&\quad +2\alpha (\overline{p}-1)L \norm{\mathbf{v}_{k}-A\mathbf{v}_{k}}\norm{A\mathbf{v}_{k}-\mathbf{x}^{*}} \\
	&\quad +2\alpha^{2}(\overline{p}-1)L^{2} \norm{\mathbf{v}_{k}-A\mathbf{v}_{k}}\norm{A\mathbf{v}_{k}-\mathbf{x}^{*}}
\end{align*}
where we used the fact that $H$ is $L$-Lipschitz from Assumption \ref{asmp:f_and_F_conditions}. Rearranging the terms we get that,
\begin{align}
&\snorm{\mathbf{v}_{k+1}-A\mathbf{v}_{k+1}}_{P_{k+1}} \notag \\
&\leq \snorm{\mathbf{v}_{k}-A\mathbf{v}_{k}}_{P_{k}} \notag \\
	&\quad + \bracket{-1 +2\alpha (\overline{p}-1)L + \alpha^{2}(\overline{p}-1)L^{2} }\snorm{\mathbf{v}_{k}-A\mathbf{v}_{k}} \notag \\
	&\quad +2\alpha (\overline{p}-1)L\bracket{1+\alpha L} \norm{\mathbf{v}_{k}-A\mathbf{v}_{k}}\norm{A\mathbf{v}_{k}-\mathbf{x}^{*}} \notag \\
	&\quad + \alpha^{2}(\overline{p}-1)L^{2} \snorm{A\mathbf{v}_{k}-\mathbf{x}^{*}} \label{eqn:V_1}
\end{align}

From \eqref{eqn:compact_W_truthful} and \eqref{eqn:compact_gradient}, the second term in \eqref{eqn:lyap} is equal to
\begin{align*}
& \snorm{A\mathbf{v}_{k+1} - \mathbf{x}^{*}} \\
&= \snorm{A\mathbf{v}_{k+1} - \mathbf{x}^{*}} = \snorm{A\overline{W}_{k}\mathbf{x}_{k} - \mathbf{x}^{*}} = \snorm{A\mathbf{x}_{k} - \mathbf{x}^{*}} \\
&= \snorm{A(1-\alpha H)\mathbf{v}_{k} - (1-\alpha H)\mathbf{x}^{*}} \\
&= \snorm{(A\mathbf{v}_{k}-\mathbf{x}^{*}) - \alpha (AH\mathbf{v}_{k} - AHA\mathbf{v}_{k} + AHA\mathbf{v}_{k} - AH\mathbf{x}^{*})} \\
&= \snorm{A\mathbf{v}_{k}-\mathbf{x}^{*}} + \alpha^{2} \snorm{AH\mathbf{v}_{k} - AHA\mathbf{v}_{k}} \\
	&\quad + \alpha^{2}\snorm{AHA\mathbf{v}_{k} - AH\mathbf{x}^{*}} \\
	&\quad -2\alpha \inp*{A\mathbf{v}_{k}-\mathbf{x}^{*}}{AH\mathbf{v}_{k} - AHA\mathbf{v}_{k}} \\
	&\quad -2\alpha \inp*{A\mathbf{v}_{k}-\mathbf{x}^{*}}{AHA\mathbf{v}_{k} - AH\mathbf{x}^{*}} \\
	&\quad + 2\alpha^{2} \inp*{AH\mathbf{v}_{k} - AHA\mathbf{v}_{k}}{AHA\mathbf{v}_{k} - AH\mathbf{x}^{*}}
\end{align*}
Note that $\norm{A} = 1$, and $\inp*{A\mathbf{v}_{k}-\mathbf{x}^{*}}{AHA\mathbf{v}_{k} - AH\mathbf{x}^{*}} \geq \mu \snorm{A\mathbf{v}_{k}-\mathbf{x}^{*}}$  then,
\begin{align*}
& \snorm{A\mathbf{v}_{k+1} - \mathbf{x}^{*}} \\
&\leq \snorm{A\mathbf{v}_{k}-\mathbf{x}^{*}} + \alpha^{2}L^{2} \snorm{\mathbf{v}_{k} - A\mathbf{v}_{k}} + \alpha^{2}L^{2}\snorm{A\mathbf{v}_{k} - \mathbf{x}^{*}} \\
	&\quad + 2\alpha L \norm{A\mathbf{v}_{k}-\mathbf{x}^{*}}\norm{\mathbf{v}_{k} - A\mathbf{v}_{k}} -2\alpha \mu \snorm{A\mathbf{v}_{k}-\mathbf{x}^{*}} \\
	&\quad + 2\alpha^{2}L^{2} \norm{\mathbf{v}_{k} - A\mathbf{v}_{k}}\norm{A\mathbf{v}_{k} - \mathbf{x}^{*}}
\end{align*}
rearranging terms we get that
\begin{align}
& \snorm{A\mathbf{v}_{k+1} - \mathbf{x}^{*}} - \snorm{A\mathbf{v}_{k}-\mathbf{x}^{*}} \notag \\
&\leq \bracket{-2\alpha \mu + \alpha^{2}L^{2} }\snorm{A\mathbf{v}_{k}-\mathbf{x}^{*}} \notag \\ 
	&\quad + 2\alpha L \bracket{1+\alpha L} \norm{A\mathbf{v}_{k}-\mathbf{x}^{*}}\norm{\mathbf{v}_{k} - A\mathbf{v}_{k}} \notag \\
	&\quad + \alpha^{2}L^{2} \snorm{\mathbf{v}_{k} - A\mathbf{v}_{k}} \label{eqn:V_2}
\end{align}
Combining the inequalities \eqref{eqn:V_1} and \eqref{eqn:V_2}, $V_{k+1}$ is bounded by
\begin{align*}
	V_{k+1} - V_{k} &\leq -\varpi^{T}_{k} \, M \, \varpi^{T}_{k}
\end{align*}
where $\varpi_{k} = \col(\norm{A\mathbf{v}_{k} - \mathbf{x}^{*}}, \norm{\mathbf{v}_{k} - A\mathbf{v}_{k}})$ and $M$ is as in \eqref{eq_M}. Since  $M$ is positive definite, then $V_{k+1} < V_{k}$, for any  $\varpi_{k} \neq \mathbf{0}$. Note that when $\varpi_{k} = \mathbf{0}$ then $\mathbf{v}_{k}$ is at consensus and $\mathbf{v}_{k} = \mathbf{x}^{*}=\mathbf{1}\otimes x^{*}$. Hence, by Theorem 13.11, \cite{Haddad},  $\mathbf{v}_{k}$ converges to $ \mathbf{1}\otimes x^{*}$ and, by continuity, from \eqref{eqn:compact_gradient}, $\mathbf{x}_{k+1} $ converges to $\mathbf{1}\otimes x^{*} -\alpha \mathcal{R}^{T}\mathbf{F}(\mathbf{1}\otimes x^{*}) = \mathbf{1}\otimes x^{*}$,  where $x^*$ is the NE.
\end{proof}

\begin{remark}
Note  that $M_{11}>0$ and $M_{22}>0$ for small $\alpha$. Additionally, $M_{11}M_{22} \propto \alpha$ and $M_{12}M_{21}\propto \alpha^{2}$. Therefore, there exists a sufficiently small $\alpha$ such that $M_{11}M_{22} > M_{12}M_{21}$ holds, hence $M$ is positive definite. 
\end{remark}

\section{Simulation} \label{sec:sim}

In this section we consider a sensor network/robot formation problem modeled as a game \cite{sensor}. The objective of each agent is to be near a subset of agents while at the same time staying close to their prescribed location. We consider a group of $96$ mobile robots in the plane, where the cost function for each agent is
\begin{align*}
	J_{i}(x_{i},x_{-i}) &= a(x) + r_{i}(x)
\end{align*}
where $a(x) = \frac{1}{2}\snorm{\frac{1}{N}\sum_{j\in\mathcal{N}} x_{j} - Q}$ is the  distance of the average 
from the target $Q$ location, and $r_{i}(x) = \sum_{j\in \mathcal{N}^{in}(\mathcal{G}_{cost})} \frac{1}{2}\snorm{x_{i}-x_{j}-d_{ij}}$ quantifies  the cost for being more than $d_{ij}$ units away from its neighbour $x_{j}$. In this example we set $Q = 0$ and $G_{cost}$ as in Figure \ref{fig:g_cost}.

\vspace{2mm}
\begin{minipage}[ht]{1\columnwidth}
\center
\begin{tikzpicture}[scale=0.6]
\def\nSize{0.4}
\def\truth{blue!60!white}
\def\liar{red!70!white}

\draw[draw, dashed, fill=green!20!white] (3.75,2.75) rectangle (8.25,7.25);

\node[draw, circle, scale=\nSize, \truth, fill=\truth] (n2) at (1,0) {};
\node[draw, circle, scale=\nSize, \truth, fill=\truth] (n3) at (2,0) {};
\node[draw, circle, scale=\nSize, \truth, fill=\truth] (n4) at (3,0) {};
\node[draw, circle, scale=\nSize, \truth, fill=\truth] (n5) at (4,0) {};
\node[draw, circle, scale=\nSize, \truth, fill=\truth] (n6) at (5,0) {};
\node[draw, circle, scale=\nSize, \liar, fill=\liar] (n7) at (6,0) {};
\node[draw, circle, scale=\nSize, \truth, fill=\truth] (n8) at (7,0) {};
\node[draw, circle, scale=\nSize, \truth, fill=\truth] (n9) at (8,0) {};

\node[draw, circle, scale=\nSize, \liar, fill=\liar] (n11) at (0,1) {};
\node[draw, circle, scale=\nSize, \truth, fill=\truth] (n12) at (1,1) {};
\node[draw, circle, scale=\nSize, \liar, fill=\liar] (n13) at (2,1) {};
\node[draw, circle, scale=\nSize, \truth, fill=\truth] (n14) at (3,1) {};
\node[draw, circle, scale=\nSize, \truth, fill=\truth] (n15) at (4,1) {};
\node[draw, circle, scale=\nSize, \truth, fill=\truth] (n16) at (5,1) {};
\node[draw, circle, scale=\nSize, \truth, fill=\truth] (n17) at (6,1) {};
\node[draw, circle, scale=\nSize, \truth, fill=\truth] (n18) at (7,1) {};
\node[draw, circle, scale=\nSize, \truth, fill=\truth] (n19) at (8,1) {};
\node[draw, circle, scale=\nSize, \liar, fill=\liar] (n20) at (9,1) {};

\node[draw, circle, scale=\nSize, \truth, fill=\truth] (n21) at (0,2) {};
\node[draw, circle, scale=\nSize, \truth, fill=\truth] (n22) at (1,2) {};
\node[draw, circle, scale=\nSize, \truth, fill=\truth] (n23) at (2,2) {};
\node[draw, circle, scale=\nSize, \truth, fill=\truth] (n24) at (3,2) {};
\node[draw, circle, scale=\nSize, \truth, fill=\truth] (n25) at (4,2) {};
\node[draw, circle, scale=\nSize, \truth, fill=\truth] (n26) at (5,2) {};
\node[draw, circle, scale=\nSize, \truth, fill=\truth] (n27) at (6,2) {};
\node[draw, circle, scale=\nSize, \truth, fill=\truth] (n28) at (7,2) {};
\node[draw, circle, scale=\nSize, \truth, fill=\truth] (n29) at (8,2) {};
\node[draw, circle, scale=\nSize, \truth, fill=\truth] (n30) at (9,2) {};

\node[draw, circle, scale=\nSize, \liar, fill=\liar] (n31) at (0,3) {};
\node[draw, circle, scale=\nSize, \truth, fill=\truth] (n32) at (1,3) {};
\node[draw, circle, scale=\nSize, \truth, fill=\truth] (n33) at (2,3) {};
\node[draw, circle, scale=\nSize, \truth, fill=\truth] (n34) at (3,3) {};
\node[draw, circle, scale=\nSize, \truth, fill=\truth] (n35) at (4,3) {};
\node[draw, circle, scale=\nSize, \truth, fill=\truth] (n36) at (5,3) {};
\node[draw, circle, scale=\nSize, \truth, fill=\truth] (n37) at (6,3) {};
\node[draw, circle, scale=\nSize, \truth, fill=\truth] (n38) at (7,3) {};
\node[draw, circle, scale=\nSize, \truth, fill=\truth] (n39) at (8,3) {};
\node[draw, circle, scale=\nSize, \truth, fill=\truth] (n40) at (9,3) {};

\node[draw, circle, scale=\nSize, \truth, fill=\truth] (n41) at (0,4) {};
\node[draw, circle, scale=\nSize, \truth, fill=\truth] (n42) at (1,4) {};
\node[draw, circle, scale=\nSize, \truth, fill=\truth] (n43) at (2,4) {};
\node[draw, circle, scale=\nSize, \truth, fill=\truth] (n44) at (3,4) {};
\node[draw, circle, scale=\nSize, \truth, fill=\truth] (n45) at (4,4) {};
\node[draw, circle, scale=\nSize, \truth, fill=\truth] (n46) at (5,4) {};
\node[draw, circle, scale=\nSize, \truth, fill=\truth] (n47) at (6,4) {};
\node[draw, circle, scale=\nSize, \truth, fill=\truth] (n48) at (7,4) {};
\node[draw, circle, scale=\nSize, \liar, fill=\liar] (n49) at (8,4) {};
\node[draw, circle, scale=\nSize, \truth, fill=\truth] (n50) at (9,4) {};

\node[draw, circle, scale=\nSize, \truth, fill=\truth] (n51) at (0,5) {};
\node[draw, circle, scale=\nSize, \truth, fill=\truth] (n52) at (1,5) {};
\node[draw, circle, scale=\nSize, \truth, fill=\truth] (n53) at (2,5) {};
\node[draw, circle, scale=\nSize, \truth, fill=\truth] (n54) at (3,5) {};
\node[draw, circle, scale=\nSize, \truth, fill=\truth] (n55) at (4,5) {};
\node[draw, circle, scale=\nSize, \truth, fill=\truth] (n56) at (5,5) {};
\node[draw, circle, scale=0.6, \truth, fill=blue!80!white] (n57) at (6,5) {};
\node[draw, circle, scale=\nSize, \truth, fill=\truth] (n58) at (7,5) {};
\node[draw, circle, scale=\nSize, \truth, fill=\truth] (n59) at (8,5) {};
\node[draw, circle, scale=\nSize, \truth, fill=\truth] (n60) at (9,5) {};

\node[draw, circle, scale=\nSize, \truth, fill=\truth] (n61) at (0,6) {};
\node[draw, circle, scale=\nSize, \truth, fill=\truth] (n62) at (1,6) {};
\node[draw, circle, scale=\nSize, \truth, fill=\truth] (n63) at (2,6) {};
\node[draw, circle, scale=\nSize, \liar, fill=\liar] (n64) at (3,6) {};
\node[draw, circle, scale=\nSize, \truth, fill=\truth] (n65) at (4,6) {};
\node[draw, circle, scale=\nSize, \truth, fill=\truth] (n66) at (5,6) {};
\node[draw, circle, scale=\nSize, \truth, fill=\truth] (n67) at (6,6) {};
\node[draw, circle, scale=\nSize, \truth, fill=\truth] (n68) at (7,6) {};
\node[draw, circle, scale=\nSize, \truth, fill=\truth] (n69) at (8,6) {};
\node[draw, circle, scale=\nSize, \truth, fill=\truth] (n70) at (9,6) {};

\node[draw, circle, scale=\nSize, \truth, fill=\truth] (n71) at (0,7) {};
\node[draw, circle, scale=\nSize, \truth, fill=\truth] (n72) at (1,7) {};
\node[draw, circle, scale=\nSize, \truth, fill=\truth] (n73) at (2,7) {};
\node[draw, circle, scale=\nSize, \liar, fill=\liar] (n74) at (3,7) {};
\node[draw, circle, scale=\nSize, \truth, fill=\truth] (n75) at (4,7) {};
\node[draw, circle, scale=\nSize, \truth, fill=\truth] (n76) at (5,7) {};
\node[draw, circle, scale=\nSize, \liar, fill=\liar] (n77) at (6,7) {};
\node[draw, circle, scale=\nSize, \truth, fill=\truth] (n78) at (7,7) {};
\node[draw, circle, scale=\nSize, \truth, fill=\truth] (n79) at (8,7) {};
\node[draw, circle, scale=\nSize, \truth, fill=\truth] (n80) at (9,7) {};

\node[draw, circle, scale=\nSize, \liar, fill=\liar] (n81) at (0,8) {};
\node[draw, circle, scale=\nSize, \truth, fill=\truth] (n82) at (1,8) {};
\node[draw, circle, scale=\nSize, \truth, fill=\truth] (n83) at (2,8) {};
\node[draw, circle, scale=\nSize, \truth, fill=\truth] (n84) at (3,8) {};
\node[draw, circle, scale=\nSize, \truth, fill=\truth] (n85) at (4,8) {};
\node[draw, circle, scale=\nSize, \truth, fill=\truth] (n86) at (5,8) {};
\node[draw, circle, scale=\nSize, \truth, fill=\truth] (n87) at (6,8) {};
\node[draw, circle, scale=\nSize, \truth, fill=\truth] (n88) at (7,8) {};
\node[draw, circle, scale=\nSize, \truth, fill=\truth] (n89) at (8,8) {};
\node[draw, circle, scale=\nSize, \liar, fill=\liar] (n90) at (9,8) {};

\node[draw, circle, scale=\nSize, \truth, fill=\truth] (n92) at (1,9) {};
\node[draw, circle, scale=\nSize, \truth, fill=\truth] (n93) at (2,9) {};
\node[draw, circle, scale=\nSize, \truth, fill=\truth] (n94) at (3,9) {};
\node[draw, circle, scale=\nSize, \truth, fill=\truth] (n95) at (4,9) {};
\node[draw, circle, scale=\nSize, \truth, fill=\truth] (n96) at (5,9) {};
\node[draw, circle, scale=\nSize, \truth, fill=\truth] (n97) at (6,9) {};
\node[draw, circle, scale=\nSize, \truth, fill=\truth] (n98) at (7,9) {};
\node[draw, circle, scale=\nSize, \liar, fill=\liar] (n99) at (8,9) {};

\path[-,line width=0.20mm]
	(n2) edge (n3)
	(n3) edge (n4)
	(n4) edge (n5)
	(n5) edge (n6)
	(n6) edge (n7)
	(n7) edge (n8)
	(n8) edge (n9)

	(n11) edge (n12)
	(n12) edge (n13)
	(n13) edge (n14)
	(n14) edge (n15)
	(n15) edge (n16)
	(n16) edge (n17)
	(n17) edge (n18)
	(n18) edge (n19)
	(n19) edge (n20)

	(n21) edge (n22)
	(n22) edge (n23)
	(n23) edge (n24)
	(n24) edge (n25)
	(n25) edge (n26)
	(n26) edge (n27)
	(n27) edge (n28)
	(n28) edge (n29)
	(n29) edge (n30)

	(n31) edge (n32)
	(n32) edge (n33)
	(n33) edge (n34)
	(n34) edge (n35)
	(n35) edge (n36)
	(n36) edge (n37)
	(n37) edge (n38)
	(n38) edge (n39)
	(n39) edge (n40)

	(n41) edge (n42)
	(n42) edge (n43)
	(n43) edge (n44)
	(n44) edge (n45)
	(n45) edge (n46)
	(n46) edge (n47)
	(n47) edge (n48)
	(n48) edge (n49)
	(n49) edge (n50)

	(n51) edge (n52)
	(n52) edge (n53)
	(n53) edge (n54)
	(n54) edge (n55)
	(n55) edge (n56)
	(n56) edge (n57)
	(n57) edge (n58)
	(n58) edge (n59)
	(n59) edge (n60)

	(n61) edge (n62)
	(n62) edge (n63)
	(n63) edge (n64)
	(n64) edge (n65)
	(n65) edge (n66)
	(n66) edge (n67)
	(n67) edge (n68)
	(n68) edge (n69)
	(n69) edge (n70)

	(n71) edge (n72)
	(n72) edge (n73)
	(n73) edge (n74)
	(n74) edge (n75)
	(n75) edge (n76)
	(n76) edge (n77)
	(n77) edge (n78)
	(n78) edge (n79)
	(n79) edge (n80)

	(n81) edge (n82)
	(n82) edge (n83)
	(n83) edge (n84)
	(n84) edge (n85)
	(n85) edge (n86)
	(n86) edge (n87)
	(n87) edge (n88)
	(n88) edge (n89)
	(n89) edge (n90)

	(n92) edge (n93)
	(n93) edge (n94)
	(n94) edge (n95)
	(n95) edge (n96)
	(n96) edge (n97)
	(n97) edge (n98)
	(n98) edge (n99)

	(n2) edge (n12)
	(n3) edge (n13)
	(n4) edge (n14)
	(n5) edge (n15)
	(n6) edge (n16)
	(n7) edge (n17)
	(n8) edge (n18)
	(n9) edge (n19)
	
	(n11) edge (n21)
	(n12) edge (n22)
	(n13) edge (n23)
	(n14) edge (n24)
	(n15) edge (n25)
	(n16) edge (n26)
	(n17) edge (n27)
	(n18) edge (n28)
	(n19) edge (n29)

	(n20) edge (n30)
	(n21) edge (n31)
	(n22) edge (n32)
	(n23) edge (n33)
	(n24) edge (n34)
	(n25) edge (n35)
	(n26) edge (n36)
	(n27) edge (n37)
	(n28) edge (n38)
	(n29) edge (n39)

	(n30) edge (n40)
	(n31) edge (n41)
	(n32) edge (n42)
	(n33) edge (n43)
	(n34) edge (n44)
	(n35) edge (n45)
	(n36) edge (n46)
	(n37) edge (n47)
	(n38) edge (n48)
	(n39) edge (n49)

	(n40) edge (n50)
	(n41) edge (n51)
	(n42) edge (n52)
	(n43) edge (n53)
	(n44) edge (n54)
	(n45) edge (n55)
	(n46) edge (n56)
	(n47) edge (n57)
	(n48) edge (n58)
	(n49) edge (n59)

	(n50) edge (n60)
	(n51) edge (n61)
	(n52) edge (n62)
	(n53) edge (n63)
	(n54) edge (n64)
	(n55) edge (n65)
	(n56) edge (n66)
	(n57) edge (n67)
	(n58) edge (n68)
	(n59) edge (n69)

	(n60) edge (n70)
	(n61) edge (n71)
	(n62) edge (n72)
	(n63) edge (n73)
	(n64) edge (n74)
	(n65) edge (n75)
	(n66) edge (n76)
	(n67) edge (n77)
	(n68) edge (n78)
	(n69) edge (n79)

	(n70) edge (n80)
	(n71) edge (n81)
	(n72) edge (n82)
	(n73) edge (n83)
	(n74) edge (n84)
	(n75) edge (n85)
	(n76) edge (n86)
	(n77) edge (n87)
	(n78) edge (n88)
	(n79) edge (n89)

	(n80) edge (n90)
	(n82) edge (n92)
	(n83) edge (n93)
	(n84) edge (n94)
	(n85) edge (n95)
	(n86) edge (n96)
	(n87) edge (n97)
	(n88) edge (n98)
	(n89) edge (n99)
;

\end{tikzpicture}
\captionof{figure}{Graph of $\mathcal{G}_{cost}$}
\label{fig:g_cost}
\end{minipage}
\vspace{0mm}

Additionally, $d_{ij}$ are the relative positions induced by Figure \ref{fig:g_cost} where edges are unit distance, i.e., $d_{4,5} = (-1,0)$ is the quadratic cost for agent $4$ not being $1$ unit to the left of $5$. The (12) red nodes on Figure \ref{fig:g_cost} are the adversarial agents and the remaining agents are the truthful agents. We assume that the adversarial agent $i$ is sending $\mathbf{y}^{ji} = \mathbf{x}^{ji} + \sigma(0,1)$ where $\sigma(0,1)$ is a vector of Gaussian noise with mean $0$ and standard deviation $1$. We assume that $\mathcal{G}_{c} = \mathcal{G}_{o}$ and there is an edge from $j$ to $i$ if $\norm{x_{i}-x_{j}}_{\infty} \leq 2$ on Figure \ref{fig:g_cost}. For example, the large blue circle in Figure \ref{fig:g_cost} communicates with all agents insider the green square.

We note that  the step sizes from Theorem \ref{thm:robust_gradient} are very conservative since they are for the worst possible communication graph switching, cf. Remark \ref{remark:conservative}. We run Algorithm 1 with the step size $\alpha = 1/40$. Figure \ref{fig:ne} shows the distance to the NE and Figure \ref{fig:pos} shows the position of a subset of the agents under Algorithm 1.
\begin{figure}[ht]
\centering
\begin{minipage}[ht]{1\columnwidth}
	\centerline{\includegraphics[width=7cm]{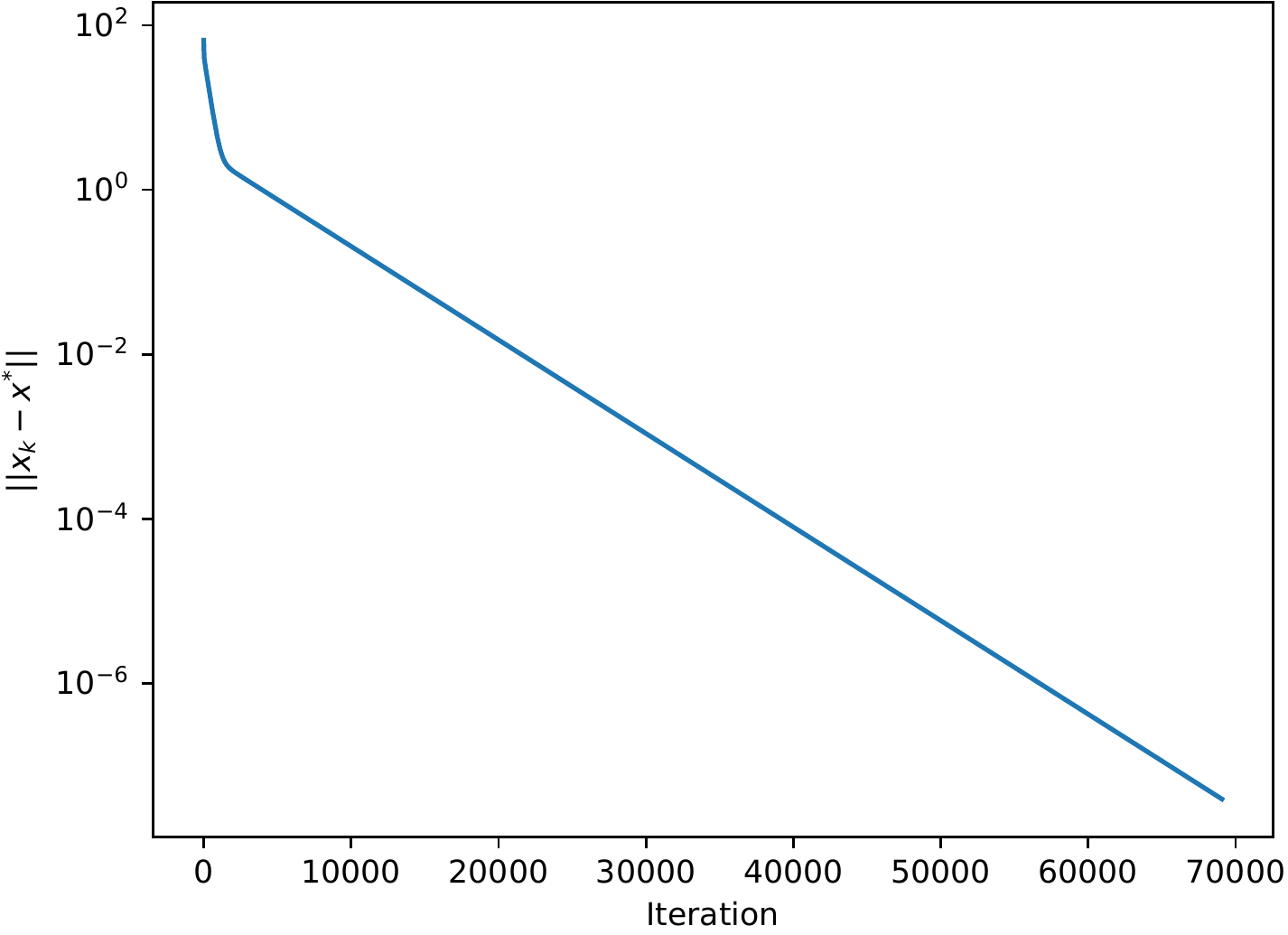}}
	\caption{Distance to the NE}\label{fig:ne}
\end{minipage}
\end{figure}

\begin{figure}[ht]
\centering
\begin{minipage}[ht]{1\columnwidth}
	\centerline{\includegraphics[width=7.5cm]{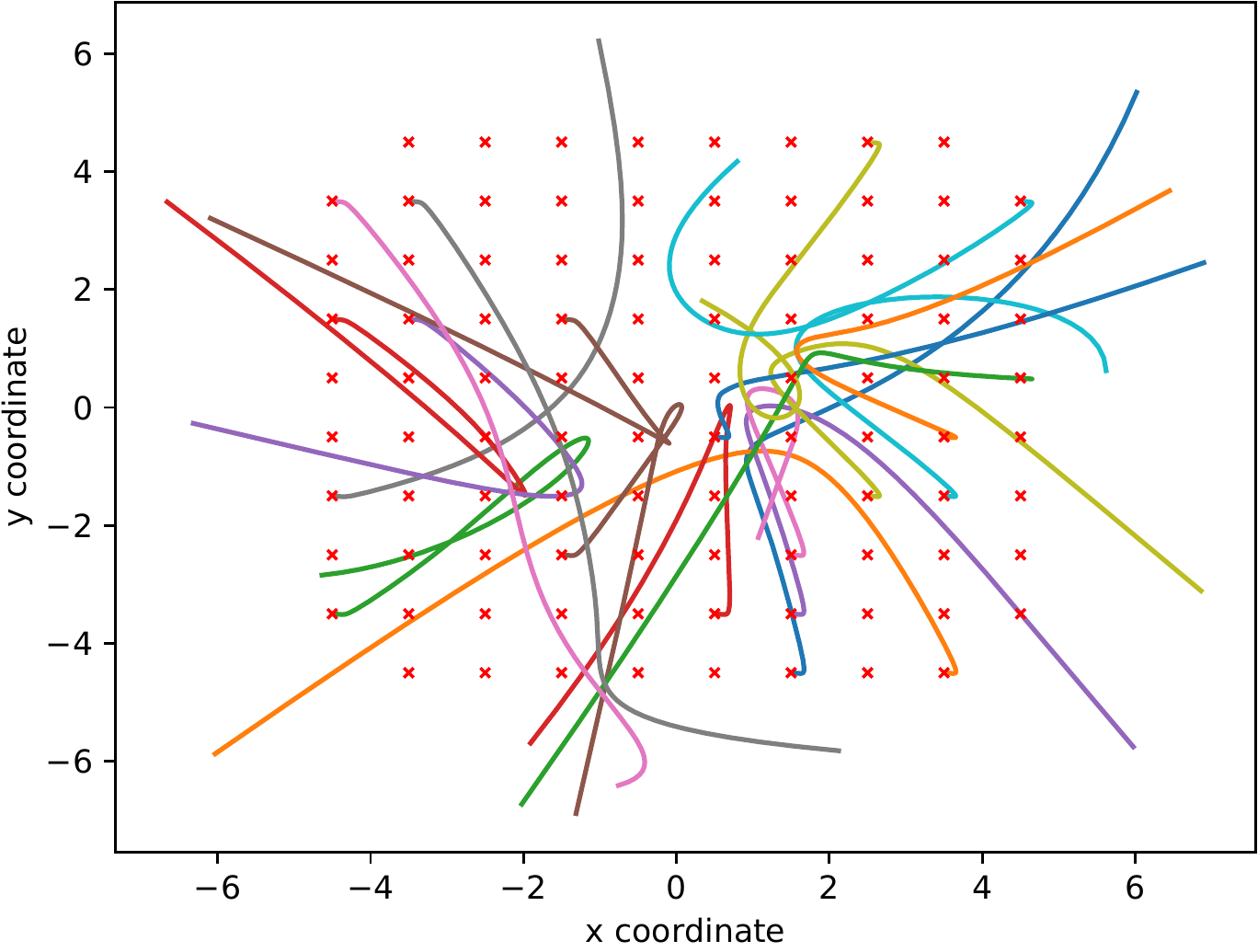}}
	\caption{Position of a subset of agents}\label{fig:pos}
\end{minipage}
\end{figure}

\section{Conclusion} \label{sec:conclusion}
In this work we considered the NE seeking problem in the partial information setting where agents are adversarial. We designed an algorithm that is robust against adversarial agents by utilizing an observation graph. The agents prune out the largest/smallest elements, take a weighted average of the remaining elements, and then preform a gradient step. Under appropriate assumptions about the number of adversarial agents and the connectivity of the graph we are able to prove convergence to the NE. A possible extension to this work would be to design a resilient GNE seeking algorithm, or relaxing some of the assumptions (e.g. Assumption \ref{asmp:adversary_dynamics}).

\bibliographystyle{IEEEtran}
\bibliography{referencesCDC}

\section*{Appendix}

\subsection*{Proof of Lemma \ref{lemma:lyapunov_ltv} }

Let $\Phi(k+s,k) \defeq \overline{W}[k+s]\overline{W}[k+s-1]\cdots \overline{W}[k]$ for all $s\geq 0$ and $\Phi(k-1,k) \defeq I$. Let $\Gamma(k+s,k) \defeq \Phi(k+s,k)^{T}\Phi(k+s,k)$. Let $P[k]\defeq (I-A)^{T}\bracket{\sum_{j=k}^{\infty}\Gamma(j-1,k)}(I-A)$. We can see that,
\begin{align*}
	& \overline{W}[k]^{T}P[k+1]\overline{W}[k] \\
	&= \overline{W}[k]^{T}(I-A)^{T}\bracket{\sum_{j=k+1}^{\infty}\Gamma(j-1,k+1)}(I-A)\overline{W}[k] \\
	&= (I-A)^{T}\overline{W}[k]^{T}\bracket{\sum_{j=k+1}^{\infty}\Gamma(j-1,k+1)}\overline{W}[k](I-A) \\
	&= (I-A)^{T}\bracket{\sum_{j=k}^{\infty}\Gamma(j-1,k) - I}(I-A) \\
	&= P[k] - (I-A)^{T}(I-A)
\end{align*}
Note that $\Gamma(j-1,k)$ has full rank because the graph $\overline{W}[k]$ is strongly connected for all $k$ and $Null(I-A) = \mathbf{1}\otimes v$ for all $v$. Therefore, the $Null(P[k]) = \mathbf{1}\otimes v$.

Next, we find an upper bound on $P[k]$, i.e., $P[k] \preceq \bar{p}I$. 	The following argument holds for all $m\in\mathcal{N}$ and $q\in\set{1,2,\dots,n_{m}}$. Notice that the elements in $\Phi(k+s,k)$ are the sum of all weighted paths from node $i$ to node $j$ about action $m$. From Lemma \ref{lemma:equiv_rooted} we know that for node $m$, the communication graph induced from $\tilde{w}^{ij}_{mq}$ and the observation graph $\mathcal{G}_{o}$, is $1$-information robust for all $k$ and the edges have weight of at least $\frac{\eta}{2}$. Therefore, there exists a path, with at most $N-1$ edges, from $m$ to every node. The element $\phi^{im}_{mq}$ in $\Phi(k+N-1,k)$, representing the path from $m$ to any node $i$ about component $q$ of action $m$, has weight of at least $\bracket{\frac{\eta}{2}}^{N-1} = 1-C$ for any iteration $k$ \footnote{If there is a path of length $b < N-1$ from $m$ to $j$, then $\Phi(k+b,k)$ will have weight $\bracket{\frac{\eta}{2}}^{b}$. Each node has a self loop with weight at least $\frac{\eta}{2}$ which accounts for the remaining $\bracket{\frac{\eta}{2}}^{(N-1)-b}$ factor after $N-1$ iterations}. Let $\tilde{\phi}^{im}_{mq}$ be an element in $\Phi(k+2(N-1),k+N-1)$ and $\hat{\phi}^{im}_{mq}$ in $\Phi(k+2(N-1),k)=\Phi(k+2(N-1),k+N-1)\Phi(k+N-1,k)$. Then,
\begin{align*}
	\hat{\phi}^{im}_{mq} &= \sum_{j=0}^{N}\tilde{\phi}^{ij}_{mq} \phi^{jm}_{mq}  = \tilde{\phi}^{im}_{mq} \phi^{mm}_{mq} + \sum_{j=0, j\neq m}^{N}\tilde{\phi}^{ij}_{mq} \phi^{jm}_{mq} \\
	&\geq \bracket{\frac{\eta}{2}}^{N-1} + \bracket{1-\bracket{\frac{\eta}{2}}^{N-1}}\bracket{\frac{\eta}{2}}^{N-1} \\
	&\geq \bracket{1-C^{2}}
\end{align*}
where we used the fact that $\phi^{mm}_{mq}=1$, $\tilde{\phi}^{im}_{mq}, \phi^{jm}_{mq} \geq \bracket{\frac{\eta}{2}}^{N-1}$ and $\sum_{j=0}^{N}\tilde{\phi}^{ij}_{mq} =1$. By an inductive argument, it can be shown that the weight $\breve{\phi}^{im}_{mq}$ in $\Phi(k+r(N-1),k)$ is at least $\bracket{1-C^{r}}$ for all $m\in\mathcal{N}$. Note that the rows of $\Phi(k+r(N-1),k)(I-A)$ are of the form,
\begin{align*}
	\begin{bmatrix}
	\breve{\phi}^{i1}_{mq}, \dots \breve{\phi}^{i,m-1}_{mq}, \breve{\phi}^{im}_{mq}-1, \breve{\phi}^{i,m+1}_{mq}, \dots, \breve{\phi}^{iN}_{mq}
	\end{bmatrix}
\end{align*}
Therefore, we know that the infinity norm is,
\begin{align*}
	&\qquad\norm{\bracket{\Phi(k+r(N-1),k)-A}}_{\infty} \\
	&= \max_{i,m}\bracket{ 1-\breve{\phi}^{im}_{mq} + \sum_{j\neq m} \breve{\phi}^{ij}_{mq}} = 2C^{r}
\end{align*}
Using this bound we know that,
\begin{align*}
	\norm{P[k]} &= \norm{\sum_{j=k}^{\infty}(I-A)^{T}\Gamma(j-1,k)(I-A)} \\
	&\leq \sum_{j=k}^{\infty}\norm{(I-A)^{T}\Gamma(j-1,k)(I-A)} \\
	&= \sum_{j=k}^{\infty}\snorm{\Phi(j-1,k)(I-A)} \\
	&\leq N\sum_{j=k}^{\infty}\snorm{\Phi(j-1,k)(I-A)}_{\infty} \\
	&= 4N\sum_{j=k}^{\infty}C^{\frac{2(j-k)}{N-1}} = \frac{4N}{1-C^{\frac{2}{N-1}}}
\end{align*}
where the second last line follows from $\norm{M}_{2} \leq \sqrt{q}\norm{M}_{\infty}$, $M\in\reals^{q\times q}$ and the fact that $\Phi$ is made of $n$ independent graphs of size $N$. \hfill \QEDclosed

\end{document}